\documentclass[10pt,a4paper]{article}

\usepackage[section]{placeins}
\usepackage{amsmath}
\usepackage{amsthm}
\usepackage{amssymb}
\usepackage{amsfonts}
\usepackage{lmodern}
\usepackage{bm}
\usepackage{mathtools}
\usepackage[T1]{fontenc}
\usepackage{framed, color}
\usepackage{dsfont}
\usepackage{graphicx}
\usepackage{hyperref}
\usepackage{latexsym}
\usepackage{mathrsfs}
\usepackage{tikz}
\usepackage{mathdots}
\usepackage{verbatim}
\usepackage{epstopdf}
\usepackage[a4paper, left=3cm, 
		     right=3cm, top=2.97cm, 
		    ]{geometry}

\theoremstyle{plain}
\newtheorem{theorem}{Theorem}[section]
\newtheorem{proposition}[theorem]{Proposition}
\newtheorem{lemma}[theorem]{Lemma}

\newtheorem{remark}[theorem]{Remark}
\newtheorem{definition}[theorem]{Definition}

\newcommand{\RR}{\mathbb R}
\newcommand{\ZZ}{\mathbb Z}
\newcommand{\TT}{\mathbb T}
\newcommand{\CC}{\mathbb C}
\newcommand{\NN}{\mathbb N}

\newcommand{\U}{\mathcal{U}}
\newcommand{\B}{\mathcal{B}}
\renewcommand{\B}{\mathcal{B}}
\renewcommand{\L}{\mathcal{L}}
\renewcommand{\S}{\mathcal{S}}
\newcommand{\tL}{\tilde{\mathcal{L}}}

\newcommand{\tT}{\tilde{T}}
\newcommand{\D}{\mathcal{D}}

\newcommand{\X}{\mathcal{X}}
\newcommand{\bu}{\bar{u}}
\newcommand{\bU}{\bar{U}}
\newcommand{\bv}{\bar{v}}

\newcommand{\bN}{\bar{N}}
\newcommand{\btheta}{\bar{\theta}}

\newcommand{\kL}{\mathfrak{L}}

\newcommand{\tb}{\tilde{b}}
\newcommand{\ul}{\underline{\lambda}}

\newcommand{\rstar}{r_*}
\newcommand{\tend}{\tau}
\newcommand{\runfile}{\texttt{runproofs.m}}
\renewcommand{\d}{\mathrm{d}}
\newcommand{\uin}{{u^{\mathrm{in}}}}
\newcommand{\vin}{{v^{\mathrm{in}}}}
\newcommand{\ds}{\displaystyle}
\newcommand{\pprod}{\mathrm{prod}}
\newcommand{\stat}{\mathrm{stat}}

\newcommand{\bydef}{\,\stackrel{\mbox{\tiny\textnormal{\raisebox{0ex}[0ex][0ex]{def}}}}{=}\,} 

\DeclareMathOperator{\diag}{diag}
\DeclareMathOperator{\bbox}{Box}

\newcommand{\Nu}{N_{u}}
\newcommand{\NL}{N_{\text{\tiny{$\L$}}}}
\newcommand{\Nin}{\Nu}

\title{A simple rigorous integrator for semilinear parabolic PDEs}
\author{Jan Bouwe van den Berg \thanks{VU Amsterdam, Department of Mathematics, Amsterdam Center for Dynamics and Computation, De Boelelaan 1081, 1081 HV Amsterdam, The Netherlands. \texttt{janbouwe@few.vu.nl}} $\quad$ Maxime Breden \thanks{CMAP, CNRS, \'Ecole polytechnique, Institut Polytechnique de
Paris, 91120 Palaiseau, France. \texttt{maxime.breden@polytechnique.edu}}} 
\date{\today}

\begin{document}

\maketitle

\begin{abstract} 
Simulations of the dynamics generated by partial differential equations (PDEs) provide approximate, numerical solutions to initial value problems. Such simulations are ubiquitous in scientific computing, but the correctness of the results is usually not guaranteed. We propose a new method for the rigorous integration of parabolic PDEs, i.e., the derivation of rigorous and explicit error bounds between the numerically obtained approximate solution and the exact one, which is then proven to exist over the entire time interval considered. These guaranteed error bounds are obtained a posteriori, using a fixed point reformulation based on a piece-wise in time constant approximation of the linearization around the numerical solution.
Our setup leads to relatively simple-to-understand estimates, which has several advantages. Most critically, it allows us to optimize various aspects of the proof, and in particular to provide an adaptive time-stepping strategy. In case the solution converges to a stable hyperbolic equilibrium, we are also able to prove this convergence, applying our rigorous integrator with a final, infinitely long timestep. We showcase the ability of our method to rigorously integrate over relatively long time intervals, and to capture non-trivial dynamics, via examples on the Swift--Hohenberg equation, the Ohta--Kawasaki equation and the Kuramoto--Sivashinsky equation. We expect that the simplicity and efficiency of the approach will enable generalization to a wide variety of other parabolic PDEs, as well as applications to boundary value problems.
\end{abstract}

\begin{center}
{\bf \small Keywords} \\ \vspace{.05cm}
{ \small Parabolic PDEs $\cdot$ Initial value problems $\cdot$ Computer-assisted proofs $\cdot$ A posteriori error estimates}
\end{center}

\begin{center}
{\bf \small Mathematics Subject Classification (2020)}  \\ \vspace{.05cm}
{\small 35K15 $\cdot$ 35K30 $\cdot$ 35K58 $\cdot$  37L65 $\cdot$ 65G20 $\cdot$ 65M15 $\cdot$ 65M70} 
\end{center}

%%%%%%%%%%%

%%%%%%%%%%%

%!TEX root = main.tex
\section{Introduction}
\label{sec:introduction}

In recent years, computer-assisted proofs have proven to be an increasingly powerful tool to study partial differential equations (PDEs) and their often intricate dynamics~\cite{BerLes15}. Numerical experiments have long played an important role in helping us
discovering and understanding new behaviors generated by PDEs, and computer-assisted proofs can sometimes be used to get mathematically rigorous and guaranteed theorems out of such simulations, especially outside of perturbative or asymptotic regimes where more traditional pen-and-paper techniques often shine. 

Going back to the pioneering works~\cite{Nak88,Plu92}, many computer-assisted proofs techniques have been developed for elliptic PDEs, see for instance~\cite{BerWil19,AriGazKoc21,LiuNakOis22,Wun22,CadLesNav24} for some recent works, as well as the book~\cite{NakPluWat19} and the references therein for a broader overview. Other types of specific solutions have also been studied with computer-assistance, such as traveling waves~\cite{AriKoc15,BreHorMcKPlu06,BerBreLesMur18,EncGomVer18,BerShe20,Cad24}, periodic solutions~\cite{Zgl04,AriKoc10bis,GamLes17,FigLla17,BerBreLesVee21}, even for some ill-posed problems~\cite{CasGamLes18}, or more recently self-similar solutions~\cite{BucCaoGom22,CheHou25,BreChu24,DahFig24,DonSch24,HouWanYan25}. For other types of specific solutions we refer to the survey~\cite{Gom19}. Although these are time-dependent solutions, in many cases the problem is reformulated as an elliptic one after using a suitable Ansatz.

A different approach is needed in order to study generic initial value problems for parabolic PDEs. In this work, we focus on scalar equations of the form
\begin{align}
	\begin{cases}
		\dfrac{ \partial u}{ \partial t} = 
		(-1)^{J+1} \dfrac{ \partial^{2J} u}{ \partial x^{2J}} + \displaystyle\sum_{j=0}^{2J-1} \dfrac{\partial^{j} g^{(j)} (u)}{\partial x^j} ,
		& t \in (0,\tau], \ x \in \TT, \\[2ex]
		u (0, x) = u^{\mathrm{in}} ( x ), & x \in \TT, 
	\end{cases}
	\label{eq:PDE}
\end{align}
where $\TT = \RR/(2\pi\ZZ)$ is the torus, $J\in\NN_{\geq 1}$ is a positive integer, $\tau>0$ determines the integration time, $g^{(j)}:\RR\to\RR$ are given polynomial functions for $j=0,\ldots,2J-1$, and $\uin: \TT \rightarrow \RR$ is a smooth function. 

\begin{remark}
We restrict to equations of the form~\eqref{eq:PDE} for the sake of simplifying the presentation, but the approach presented in this paper is applicable in a broader setting. In particular, it generalizes to systems in a completely straightforward manner, as well as to PDEs posed on higher dimensional tori.
Moreover, as soon as the initial data has some symmetries (e.g., is even or odd) that are preserved by the PDE, then these symmetries can be incorporated into our setup and we can analyze dynamics restricted to a symmetry-induced invariant subspace. This means we are also able to deal with PDEs posed on a bounded interval with homogeneous Dirichlet or Neumann boundary conditions, as will be shown in examples. 
\end{remark}

Within the computer-assisted proof community, techniques allowing to rigorously solve initial value problems of the form~\eqref{eq:PDE} are often described as \emph{rigorous PDE integrators}. One of our main motivation for developing rigorous integrators is to then be able to study boundary value problems in time, in conjunction with rigorous enclosures for local stable and unstable manifolds that have been developed recently~\cite{MirRei19,BerJaqMir22}, in order to obtain connecting orbits in parabolic PDEs. Rigorous PDE integrators can also be used to rigorously compute Poincaré maps, see~\cite{CAPD} and the references therein.

One class of rigorous PDE integrators, which uses completely different techniques compared to the ones discussed in this manuscript, is the one developed in~\cite{Zgl04,Zgl10}. It is based on the validated integration of a finite dimensional system of ODEs using Lohner-type algorithms developed in 
\cite{Zgl02}, and on the notion of self-consisted bounds introduced in~\cite{ZglMis01}. This methodology was further developed in~\cite{Cyr14,WilZgl25}, and has been used to rigorously study, among other things, globally attracting solutions in the one dimensional Burgers equation~\cite{CyrZgl15}, heteroclinic connections in the one-dimensional Ohta-Kawasaki model~\cite{CyrWan18} and recently to prove chaos in the Kuramoto-Sivashinky equation~\cite{WilZgl20,WilZgl24}.

A second family of rigorous PDE integrators, to which the approach proposed in this paper belongs, is based on a fixed point reformulation which allows to prove, a posteriori, that a true solution exists in a small neighborhood of a numerically computed approximation. In order to first describe the main ideas underlying our approach informally, let us consider an abstract evolution equation (still assumed to be of parabolic type)
\begin{align}
\label{eq:abstract_evolution}
\begin{cases}
\partial_t u = F(u),\qquad t\in(0,\tau], \\
u(t=0) = \uin.
\end{cases}
\end{align}
Since we aim to first keep the discussion simple, we purposely do not yet specify functions spaces and norms. However, as the reader might have guessed, when turning the following informal discussion into actual estimates, choosing proper spaces and norms will prove critical.

A natural way of rewriting the evolution equation~\eqref{eq:abstract_evolution} as a fixed point problem is to consider a (possibly time-dependent) linear operator $\L(t)$ generating a $C_0$-semigroup, and $\gamma(t,u):= F(u) - \L(t) u$, which yields
\begin{align}
\label{eq:abstract_evolution_split}
\begin{cases}
\partial_t u = \L(t) u + \gamma(t,u),\qquad t\in(0,\tau] \\
u(t=0) = \uin.
\end{cases}
\end{align}
We denote by $\U=\U(t,s)$ the evolution operator (sometimes also called solution operator) associated to $\L$, i.e., the operator such that $\U(t,s)\varphi$ for all $t\geq s\geq 0$ is the solution of
\begin{align*}
\begin{cases}
\partial_t u = \L(t)u,\qquad t>s \\
u(t=s) = \varphi.
\end{cases}
\end{align*}
The solution $u$ of~\eqref{eq:abstract_evolution} then corresponds to a fixed point of the operator $T$ given by
\begin{align}
\label{eq:T_general}
T(u)(t) = U(t,0) \uin + \int_0^t U(t,s) \gamma(s,u(s)) \d s,\qquad t\in[0,\tau].
\end{align}
Given an approximate solution $\bu$ of~\eqref{eq:abstract_evolution}, our goal is to prove that there exists an exact solution $u$ near $\bu$ and to provide an explicit and guaranteed error bound between $u$ and $\bu$, by showing that an operator $T$ of the form~\eqref{eq:T_general} is a contraction on a small neighborhood of $\bu$. Note that local well-posedness is known for equations of the form~\eqref{eq:PDE}, but finite-time blow up might occur, hence even the mere existence of a solution up to time $\tau$ is not guaranteed a priori.

In the construction of $T$ in~\eqref{eq:T_general}, the choice of $\L$ is critical. Indeed, for an arbitrary $\L$ there is no reason for $T$ to be contracting near $\bu$, especially if $\tau$ is somewhat large. However, since the Fréchet derivative of $T$ at $\bu$ is
\begin{align*}
DT(\bu)h(t) = \int_0^t U(t,s) D_u\gamma(s,\bu(s))h(s) \d s,\qquad t\in[0,\tau],
\end{align*}
a natural choice is to take $\L(t)=DF(\bu(t))$, because it yields $D_u\gamma(t,\bu(t)) = 0$ for all $t$, hence $DT(\bu) = 0$. In that case, provided $\bu$ is an accurate enough approximate solution, i.e., that $T(\bu)-\bu$ is small enough, $T$ should indeed be a contraction on a small neighborhood of $\bu$. Nonetheless, actually proving this requires getting somewhat sharp quantitative estimates on the evolution operator $\U$, which is a challenging task, especially because $\L(t)=DF(\bu(t))$ depends on time in a non-trivial way via the approximate solution $\bu$. Despite these challenges, this approach was successfully used in~\cite{TakMizKubOis17,TakLesJaqOka22,DucLesTak25}.\footnote{To be precise, these works take advantage of the parabolic/dissipative structure of~\eqref{eq:PDE} to use a simple approximation of $DF(\bu(t))$ for the high frequencies, similar to the one that is going to be used in this paper, but for the low frequencies they do indeed take $\L(t)=DF(\bu(t))$.} Very similar ideas, but presented with a different viewpoint, can be found in~\cite{HasKimMinNak19,HasKinNak20}, see also~\cite[Chapter 5]{NakPluWat19}. 

A somewhat radical alternative, which in fact corresponds to what was done in some of the earliest rigorous PDE integrators~\cite{Nak91,AriKoc10bis}, is to simply take $\L=(-1)^{J+1} \tfrac{ \partial^{2J}}{ \partial x^{2J}}$, or $\L = DF(0)$. In those cases, $\L$ no longer depends on time, and the evolution operator is simply $\U(t,s) = e^{(t-s)\L}$, which is therefore much easier to study (especially since such an $\L$ is diagonal in Fourier space). However, the downside is that we can only expect the corresponding fixed point operator $T$ to be contracting near $\bu$ when $\tau$ is small.

In this paper we take an intermediate path, namely we take for $\L$ a piece-wise constant approximation of $DF(\bu(t))$. The rational behind this choice is the following. We recall that, when $\L$ is taken equal to $DF(\bu(t))$, we have $DT(\bu) = 0$. However, this is much stronger than what we actually need in order to apply the contraction mapping, namely that $\left\Vert DT(u)\right\Vert \leq C < 1$ for all $u$ in some small neighborhood of $\bu$. Therefore, there is some room to play with in the definition of $\L$, which should be used to obtain an evolution operator $\U$ which is as easy to analyze as possible, and our choice of $\L$ is guided by these considerations.
Indeed, provided the subdivision in time is fine enough, choosing $\L$ to be a piecewise constant approximation of $DF(\bu(t))$ will cause the map $T$ to still be contracting around $\bu$ for long integration times $\tau$. Moreover, by having $\L$ piece-wise time-independent we retain the ability to get an explicit description of the evolution operator $\U$, for which estimates are then much easier to obtain. 

The simplicity of the setup has several concrete advantages. First, the estimates are relatively straightforward to understand and to code. This makes it easier to focus on optimizing critical aspects of the bounds and the implementation. One crucial example is determining a time grid (on which the operator $\L$ is piece-wise constant) that is favourable for the proof to go through, see Section~\ref{sec:parameters}. This allows us to introduce an adaptive time-stepping strategy, and to automatically select some discretization parameters on each subinterval, which also contributes to getting computationally cheaper proofs. This is of particular importance when proving long time segments, as well as when applying the integrator to study boundary value problems (discussed further below). Another benefit of the simplicity is that we believe it will be possible to extend the method to PDEs outside the family~\eqref{eq:PDE}, for example reaction-diffusion systems, multiple spatial dimensions, and equations with more complicated nonlinear lower-order terms.

\medskip

We note that in~\cite{BerBreShe24} we already did take for $\L$ a piece-wise constant approximation of $DF(\bu(t))$, but then we did not directly proceed to prove that the corresponding fixed point operator $T$ defined in~\eqref{eq:T_general} was a contraction near $\bu$. Instead, we first tried  to compensate for the fact that our $\L$ was not exactly equal to $DF(\bu(t))$ by first turning the fixed point problem into a zero finding problem $\tilde{F}(u) = u - T(u)$, and then went back to a fixed point problem via the Newton-Kantorovich approach, using an approximate inverse $A$ of $D \tilde{F}(\bu)$ and considering $\tT(u) = u - A\tilde{F}(u)$. 
However, we are now convinced that this detour was ill-advised, for several reasons. First, this reformulation is in fact not necessary, since, as explained above, the fixed point operator $T$ from~\eqref{eq:T_general} should already be contracting if the subdivision used to define $\L$ is taken fine enough. Second, in order to use a non-trivial approximate inverse $A$, one needs a finite dimensional projection, i.e., to not only truncate Fourier modes but also to choose a projection in time. We need such a projection anyway to find the approximate solution, so this is the one we used: the finite dimensional projection in time is defined as taking an interpolation polynomial at Chebyshev nodes. However, once this projection is used for defining the approximate inverse, it also has to be incorporated in the norm, which has to be split between a finite part in time and a tail part in time (see~\cite[Section 3.2]{BerBreShe24} for details). This splitting then leads to much more involved space-time estimates, which turn out not to be very sharp. In contrast, in the present work all the critical estimates for showing the contractivity of $T$ boil down to weighted $\ell^1$ operator norm bounds in Fourier space only, for which simpler and very sharp estimates can be obtained.
Finally, the computational costs associated with the construction and storage of the approximate inverse $A$ are often the practical bottleneck of such computer-assisted proofs. Therefore, getting rid of the approximate inverse significantly increases the applicability of the method. In many situations this is not possible, because the approximate inverse is the only reason why the fixed point operator we study should be contracting. However, for initial value problems specifically, we can in fact build contracting fixed point operators without using a large approximate inverse, as was also noticed in~\cite{DucLesTak25}. 

The rigorous integrator introduced in the current paper outcompetes our previous attempt from~\cite{BerBreShe24} in every aspect. Not only is it much faster and uses much less memory, see Section~\ref{sec:SH} for specific numbers, it also allows us to prove solutions of the initial value problem for significantly longer integration times and solutions exhibiting more interesting dynamical features. Examples can be found in Theorem~\ref{th:OK} below and in Section~\ref{sec:examples}.

As already mentioned, there are many situations where computing an approximate inverse of (a finite dimensional projection) of the linearized problem is a necesary step in a computer-assisted proof. Indeed, when the choice of $\L$ is less sophisticated (e.g., based on the leading order term in the PDE only), then introducing the approximate inverse $A$ does make sense. 
We note that yet another rigorous integrator, which needs an approximate inverse but has the advantage of being fully spectral, was proposed in~\cite{CyrLes22} and significantly improved in~\cite{CadLes25}.

\medskip

Our main motivation to develop this rigorous integrator for initial value problems of PDEs, is to study \emph{boundary} value problems, in particular those associated to saddle-to-saddle connecting orbit problems. As will be shown in an upcoming work, the fact that we avoid using a large approximate inverse can mostly be preserved when using the rigorous integrator to solve saddle-to-saddle connecting orbit problems for PDEs of the form~\eqref{eq:PDE}.

The slightly simpler situation of orbits converging to a stable equilibrium is also of interest, and of great relevance for many applications. In the context of computer-assisted proofs, what is often done is to first validate the steady state of interest independently, then prove it is stable, and subsequently derive an explicit neighborhood of the steady state such that any solution entering that neighborhood converges to the steady state. This was first accomplished in~\cite{Zgl02bis}, and then in many subsequent works, sometimes with slightly different techniques. After that, one uses a rigorous integrator to prove that an orbit possibly starting far away from the stable equilibrium enters this neighborhood, see~\cite{DucLesTak25,CadLes25} for some recent examples. In all these works, the stable steady state under consideration is hyperbolic, but this strategy can also be adapted in the presence of a center manifold~\cite{JaqLesTak22}. These tools can also sometimes be combined with a priori dissipation estimates in order to conclude that a steady state is globally attractive~\cite{Cyr15}.

In this work, we show that, in the presence of a locally attracting hyperbolic stable steady state, our rigorous integrator is sufficiently robust to allow us to rigorously integrate all the way to $\tau=+\infty$. Moreover, we show that, if this integration step to $\tau=+\infty$ is successful, then we automatically get the existence of a stable steady state, together with an explicit neighborhood in which all solutions converge towards the steady state, 
and an explicit spectral gap. As mentioned above, we emphasize that each of these results could have been obtained independently with existing techniques. However, the fact the existence and local stability of a steady state can be obtained ``for free'' from our rigorous integration is at the very least convenient in practice, and is also a good confirmation that our rigorous integrator incorporates some important dynamical features of the flow.

Below is an example of the kind of result we can obtain with our approach, applied to the Ohta--Kawasaki equation with homogeneous Neumann boundary conditions
\begin{align}
		\label{eq:OK}
		\begin{cases}
			\dfrac{ \partial u }{ \partial t} = 
			-\displaystyle\frac{1}{\gamma^2}\dfrac{ \partial^{4} u}{ \partial x^{4}} -\dfrac{ \partial^{2} }{ \partial x^{2}}(u-u^3) -\sigma(u-m)
			, & (t,x) \in (0,\infty) \times (0,L), \\[2ex]
			\dfrac{ \partial u}{ \partial x}(t,0)=\dfrac{ \partial u}{ \partial x}(t,L)=\dfrac{ \partial^{3} u}{ \partial x^{3}}(t,0)=\dfrac{ \partial^{3} u}{ \partial x^{3}}(t,L)=0,  & t \in [0,\infty),\\[2ex]
			u (0, x ) = \uin ( x ), & x \in [0,L], 
	\end{cases}
	\end{align}
which models the evolution of diblock copolymer melts~\cite{OhtKaw86}.
The following theorem rigorously validates the approximate solution depicted in Figure~\ref{fig:OK} on the \emph{infinite} time interval $[0,\infty)$.
We note that the solution exhibits interesting dynamics with a long transient, likely spending time near a saddle point, before progressing to the stable equilibrium. The rigorous integrator introduced in the current paper provides a proof that was clearly out of reach of our previous approach in~\cite{BerBreShe24}.
\begin{figure}
\centering
\includegraphics[scale=0.45]{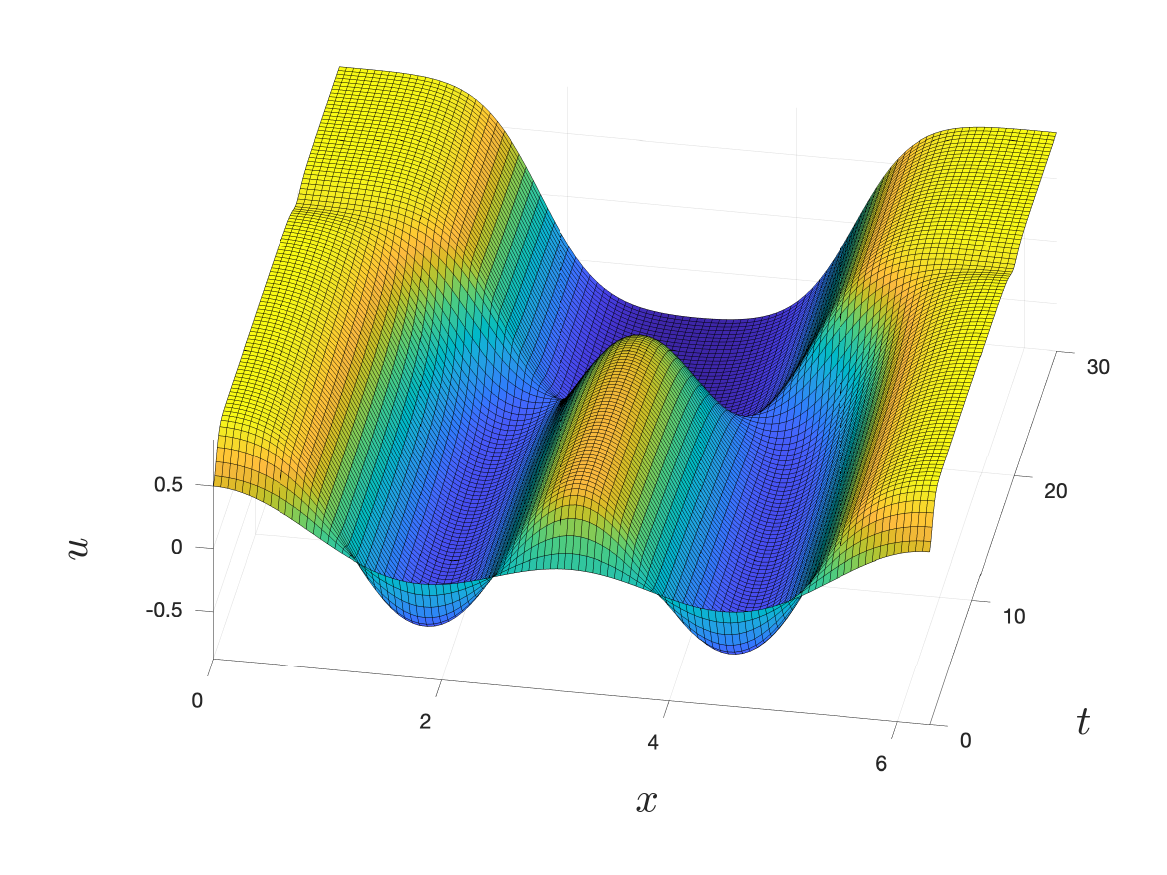}
\caption{The global approximate solution $\bu$ of the Ohta--Kawasaki equation~\eqref{eq:OK} which has been validated in Theorem~\ref{th:OK}, represented on the time interval $[0,30]$.}
\label{fig:OK}
\end{figure}
\begin{theorem}
\label{th:OK}
Consider the Ohta--Kawasaki equation~\eqref{eq:OK}	
with $\gamma = \sqrt{8}$, $\sigma=1/5$, $m=1/10$, $L=2\pi$, and $\uin(x) = m+\frac{2}{10}\cos\left(\frac{2\pi x}{L}\right)+\frac{2}{10}\cos\left(\frac{4\pi x}{L}\right)$. The solution $u=u(t,x)$ satisfies
\begin{align*}
\sup_{t\in [0,\infty)} \sup_{x\in [0,L]} \vert u(t,x)-\bu(t,x) \vert \leq 9\times 10^{-7},
\end{align*}
%where $\bu=\bu(t,x)$ is an approximate solution whose precise description in terms of Chebyshev$\times$Fourier coefficients can be downloaded
where the precise description of the approximate solution $\bu=\bu(t,x)$ in terms of Chebyshev$\times$Fourier coefficients can be downloaded
 at~\cite{integratorcode_new}. This approximate solution $\bu$ is represented in Figure~\ref{fig:OK} for $t\in[0,\tau]$ with $\tau=30$, and satisfies $\bu(t,x) = \bu(\tau,x)$ for all $t\geq \tau$ and all $x\in[0,L]$. Moreover, there exists a stationary solution $u^\stat=u^\stat(x)$ of the Otha--Kawasaki equation, such that
\begin{align*}
\sup_{x\in [0,L]} \vert u^\stat(x)-\bu(\tau,x) \vert \leq 9\times 10^{-7}.
\end{align*}
Finally, the solution $u$ of~\eqref{eq:OK} converges to $u^\stat$ as $t\to\infty$, and $u^\stat$ is locally asymptotically stable. 
\end{theorem}

\medskip

The remainder of the paper is organized as follows. In Section~\ref{sec:preliminaries}, we introduce the relevant function spaces, fix some notations, and recall a fixed point theorem that will prove convenient to rigorously validate solutions of~\eqref{eq:PDE} via the operator $T$ from~\eqref{eq:T_general}. In Section~\ref{sec:onedomain}, we then first deal with the case where the operator $\L$ approximating $DF(\bu)$ is taken constant over the whole time-interval. We provide there an explicit construction of $\L$, together with computable bounds satisfying the assumptions of the fixed point theorem. The general case, where $\L$ is only taken piece-wise constant, is then presented in Section~\ref{sec:domaindecomposition}, with adaptive stepsize and other parameter optimization discussed in Section~\ref{sec:parameters}. The extension to $\tau=+\infty$, and how this implies the existence of a locally stable steady state, is presented in Section~\ref{sec:infinity}. Various examples, including some comparison with~\cite{BerBreShe24}, are then given in Section~\ref{sec:examples}.

All the computations presented in this paper have been implemented in Matlab,
using the Intlab package~\cite{Rum99} for interval arithmetic. The computer-assisted
parts of the proofs can be reproduced using the code available at~\cite{integratorcode_new}.
%!TEX root = main.tex
\section{Preliminaries}
\label{sec:preliminaries}
In this section we introduce notation and provide the necessary background for the tools
used in this paper. 

\subsection{Sequence space of Fourier coefficients}
\label{sec:preliminaries_Fourier}
Our phase space will be described in terms of geometrically decaying Fourier sequences
\begin{align*}
	\ell^{1}_{\nu} \bydef  \biggl\{ a \in \CC^{\ZZ} : 
				\sum_{n\in\ZZ} \left \vert a_{n} \right \vert \nu^{\vert n\vert} < \infty \biggr\} ,
\end{align*}
endowed with the norm 
$\left \Vert a \right \Vert_{\ell^1_\nu} = 
\sum_{n\in\ZZ}  \left \vert a_{n} \right \vert \nu^{\vert n\vert}$,
where $\nu\geq 1$ is some decay rate to be chosen later. 
\begin{remark}[Notation]
\label{rem:identification}
In order to simplify the notation, in the rest of this paper we will use the same symbol to denote a function and its (discrete) Fourier transform, that is we write
\begin{equation*}
a(x) = \sum_{n\in\ZZ} a_{n} e^{inx}.
\end{equation*}
It should be clear from the context whether $a$ denotes the function $x\mapsto a(x)$ or the sequence of Fourier coefficients $a=\left(a_n\right)_{n\in\ZZ}$.
\end{remark}
Recall that the Fourier coefficients of the product of two Fourier expansions is given by
the discrete convolution: when $a$ and $b$ are $2\pi$-periodic functions given by
$a(x) = \sum_{n\in\ZZ} a_{n} e^{inx}$
and \
$b(x) = \sum_{n\in\ZZ} b_{n} e^{inx}$,
then
\begin{align*}
	\left( ab \right) (x) = 
	 \sum_{n\in\ZZ} \left( a \ast b\right)_{n} e^{inx},  
\qquad\text{where}\qquad
	\left(  a \ast b \right)_{n} \bydef  \sum_{m \in \ZZ}  a_m b_{n-m},
	\quad n \in \ZZ,
\end{align*}
and these formal computations are justified as soon as $a$ and $b$ are smooth enough, which is the case if their Fourier coefficients belong to $\ell^1_\nu$ for some $\nu\geq 1$. We also recall that the discrete convolution gives $\ell^{1}_{\nu}$ a Banach algebra structure.
\begin{lemma}
	\label{lem:BanachAlgebra}
For all $a,b\in\ell^1_\nu$, $a\ast b \in\ell^1_\nu$ and
$\left \Vert a\ast b \right \Vert_{\ell^1_\nu} \leq \left \Vert a \right \Vert_{\ell^1_\nu} \left \Vert b \right \Vert_{\ell^1_\nu}$.
\end{lemma}

\begin{remark}[Notation, continued]
\label{rem:absg}
Following Remark~\ref{rem:identification}, given a polynomial function $g:\RR\to\RR$ and $a\in\ell^1_\nu$, we still use $g(a)$ to denote the sequence of Fourier coefficients of the function $g(a)$. For instance, if $g(x) = x^3 - 2x$ and $a\in\ell^1_\nu$, then $g(a) \bydef a\ast a\ast a - 2a$. 

Moreover, for a polynomial function $g:\RR\to\RR$ written
$g(x) = \sum_{k=0}^K g_k x^k$, 
we denote by $\vert g\vert$ the polynomial given by
$\vert g\vert (x) \bydef \sum_{k=0}^K \vert g_k\vert x^k$.
\end{remark}

The Banach algebra property of Lemma~\ref{lem:BanachAlgebra} readily yields the following statement, which will prove useful in the sequel.
\begin{lemma}
\label{lem:abs_pol}
Let $a,b\in\ell^1_\nu$ and $g:\RR\to\RR$ a polynomial function. Then
\begin{align*}
\left\Vert g(a+b)\right\Vert_{\ell^1_\nu} \leq  \vert g\vert \left(\left\Vert a\right\Vert_{\ell^1_\nu} + \left\Vert b\right\Vert_{\ell^1_\nu}\right).
\end{align*}
\end{lemma}

We now introduce projections allowing us to work with finitely many Fourier modes.
\begin{definition}[Truncation of phase space]
	Let $N \in \NN$ be a truncation parameter. 
	The projection $\Pi^{\leq N} : \ell^1_\nu \rightarrow \ell^1_\nu$ is defined by
	\begin{align*}
		\left( \Pi^{\leq N} ( a ) \right)_{n} \bydef  
		\begin{cases}
			a_{n}, &  0 \leq \vert n\vert \leq N, \\
			0, & \vert n\vert > N. 
		\end{cases}
	\end{align*} 
	Furthermore, we set $\Pi^{>N} \bydef  I - \Pi^{\leq N}$, where $I$ is the identity on $\ell^1_\nu$. In the sequel, we frequently identify elements of $\Pi^{\leq N}\ell^1_\nu$ with vectors in $\CC^{2N+1}$, and, given a linear operator $L$ on $\ell^1_\nu$, we also identify $\Pi^{\leq N} L\, \Pi^{\leq N}$ with a $(2N+1)\times(2N+1)$ matrix. 
\end{definition}
We end this section with a final notation.
\begin{definition}[Spatial derivative]
We denote by $\D$ the Fourier transform of $\dfrac{\partial}{\partial x}$, i.e.,
\begin{equation*}
\left(\D a\right)_n \bydef  in a_n, \qquad\text{for all }a\in\ell^1_\nu.
\end{equation*}
\end{definition}

\subsection{Function space}
\label{sec:preliminaries_Xnu}

We are first going to solve~\eqref{eq:PDE} in a space which can be identified with a subspace of $C([0,\tau],\ell^1_\nu)$, the space of continuous functions from $[0,\tau]$ into $\ell^{1}_{\nu}$. 

\begin{definition}
\label{def:Xnu}
Let $\nu\geq 1$. We set
\begin{align*}
\X_\nu \bydef \left\{ u = \left(u_n\right)_{n\in\ZZ}\in \left(C([0,\tau],\CC)\right)^\ZZ,\ \left\Vert u\right\Vert_{\X_\nu} \bydef \sum_{n\in\ZZ} \left\Vert u_n\right\Vert_{C^0} \nu^{\vert n\vert} < \infty \right\},
\end{align*}
where, for any continuous function $v:[0,\tau] \to \CC$, we have
$\left\Vert v\right\Vert_{C^0} \bydef \max_{t\in [0,\tau]} \left\vert v(t)\right\vert$.
In particular, for any $u\in\X_\nu$ and $t\in[0,\tau]$, $u(t)$ can be interpreted as an element of $\ell^1_\nu$. Any linear operator on $\ell^1_\nu$ can be extended to a \emph{constant} (in time) linear operator on $\X_\nu$.
\end{definition}
\noindent Note that $\left(\X_\nu,\left\Vert \cdot\right\Vert_{\X_\nu}\right)$ is a Banach space. Also, we already warn the reader that, while this exact space will indeed be used in Section~\ref{sec:onedomain} in the specific case where $\L$ is taken constant, a slightly more involved space will in fact be introduced in Section~\ref{sec:domaindecomposition} when $\L$ is only piece-wise constant.

\begin{remark}[Notation]
Every time an operation defined on scalars or scalar functions is applied to an element of $\X_\nu$, this operation must be understood component-wise. For instance, for any $u$ in $\X_\nu$ and $t\in[0,\tau]$,
\begin{align*}
\vert u(t)\vert \bydef \left(\left\vert u_n(t)\right\vert\right)_{n\in\ZZ}\qquad \text{and}\qquad \left\Vert u\right\Vert_{C^0} \bydef \left(\left\Vert u_n\right\Vert_{C^0}\right)_{n\in\ZZ}.
\end{align*} 
In particular, note that $\left\Vert u\right\Vert_{\X_\nu} = \left\Vert\left\Vert u\right\Vert_{C^0}\right\Vert_{\ell^1_\nu}$.

The projections $\Pi^{\leq N}$ and $\Pi^{> N}$, as well as the spatial derivative $\D$, also extend in a natural way to $\X_\nu$, by being applied at each time. That is, for any $u\in\X_\nu$ and $t\in[0,\tau]$, $\left(\D u\right)(t)\bydef \D\left(u(t)\right)$.
\end{remark}

Following Remark~\ref{rem:identification}, every $u$ in $\X_\nu$ can be identified with a function $u:[0,\tau]\times\TT\to \CC$, written
\begin{align*}
u(t,x) = \sum_{n\in\ZZ} u_n(t) e^{inx}.
\end{align*}
Such a function $u$ is continuous, and we have
\begin{align*}
\max_{\substack{ t\in[0,\tau] , x \in \TT }} \left\vert u(t,x)\right\vert	\leq \left\Vert u\right\Vert_{\X_\nu}.
\end{align*}
Moreover, if $\nu>1$, the function $u$ is analytic in the second variable $x$.

\subsection{Approximation in time}
\label{sec:interp}

In time, we use approximate solutions that are piece-wise polynomials (polynomials on each subdomain in time), written in the Chebyshev basis. In particular, we often bound $C^0$-norms in time from above using that 
\begin{align}
\label{eq:C0time}
\max_{t\in[-1,1]} \left\vert \sum_{k=0}^K p_k T_k(t) \right\vert \leq \sum_{k=0}^K \vert p_k\vert,
\end{align}
where $T_k$ are Chebyshev polynomials of the first kind.

When deriving estimates to prove that the fixed point operator $T$ of the form~\eqref{eq:T_general} is a contraction, we will also need explicit and sharp upper-bounds for quantities like 
$\left\Vert \varphi \right\Vert_{C^0}$, where $\varphi$ is a potentially complicated but known smooth function, that we can rigorously evaluate at any given point. In such cases, we estimate, for some $\tilde{K}\in\NN$ large enough,
\begin{align}
\label{eq:C0interp}
\left\Vert \varphi \right\Vert_{C^0} \leq  \left\Vert P_{\tilde{K}}\varphi \right\Vert_{C^0} + \left\Vert \varphi - P_{\tilde{K}}\varphi \right\Vert_{C^0},
\end{align}
where $P_{\tilde{K}}\varphi$ is the Lagrange interpolation polynomial of $\varphi$ at $\tilde{K}+1$ Chebyshev nodes. This interpolation polynomial itself can be computed explicitly from evaluations of $\varphi$, which then yields a computable bound for $\left\Vert P_{\tilde{K}}\varphi \right\Vert_{C^0}$ (for instance using~\eqref{eq:C0time}). We then need to obtain explicit bounds for the interpolation error $\left\Vert \varphi - P_{\tilde{K}}\varphi \right\Vert_{C^0}$ (see, e.g.,~\cite{Tre13}). Provided $\tilde{K}$ is taken large enough, $\left\Vert \varphi - P_{\tilde{K}}\varphi \right\Vert_{C^0}$ should be small compared to $\left\Vert P_{\tilde{K}}\varphi \right\Vert_{C^0}$, and this procedure then yields a sharp upper-bound for $\left\Vert \varphi \right\Vert_{C^0}$. For details on how to get the required interpolation error estimates in our context, see~\cite[Section 2.2 and Appendices B and C]{BerBreShe24}.

\subsection{A fixed point Theorem}
\label{sec:NewtonKantorovich}

Let $M\in\NN_{\geq 1}$, $(X^m,\|\cdot\|_{X^m})_{m=1}^M$ be Banach spaces, $X=\Pi_{m=1}^M X^m$ the product space, and $\pi^m:X\to X^m$ the projections onto the components. Let $\rstar = (\rstar^m)_{m=1}^M \in \RR_{>0}^M$ and $\bar{x} \in X$. For any $r\in \RR_{>0}^M$, we define $\bbox(\bar{x},r) = \{ x \in X :  \left\Vert\pi^{m}(x-\bar{x}) \right\Vert_{X^m} \leq r^m \text{ for } 1\leq m\leq M\}$. We consider in this section a map $T \in C^1(\bbox(\bar{x},r),X)$. For $r,\rstar \in \RR_{>0}^M$ we say that $r \leq \rstar$ if $r^m \leq \rstar^m$ for all $1\leq m\leq M$. Finally, we denote partial Fr\'echet derivatives by $D_i$. 

The following statement, based on the Banach fixed point Theorem, provides explicit conditions under which $T$ has a unique fixed point in $\bbox(\bar{x},r)$ (for some explicit $r$). Many similar versions of this theorem have been used in the last decades for computer-assisted proofs. This specific instance is the one introduced in~\cite[Theorem 2.12]{BerBreShe24}. Its main feature is that the norm on the product space $X$ is not chosen a priori, and that we do no try to choose the set $\bbox(\bar{x},r)$ as a ball for this norm, which provides some extra flexibility when trying to prove that $T$ is a contraction, see~\cite{BerBreShe24} for a more in-depth discussion.

\begin{theorem}
\label{thm:NewtonKantorovich}
Assume that $Y^m \geq 0$, $Z^m_i \geq 0$, $W^m_{ij} \geq 0$ for $1 \leq  i,j,m \leq M$ satisfy
\begin{alignat}{2}
	\left\Vert\pi^m (T(\bar{x})-\bar{x}) \right\Vert_{X^m} &\leq Y^m, \label{e:def_Y}\\
	\left\Vert\pi^m D_i T(\bar{x}) \right\Vert_{B(X^i,X^m)} &\leq Z^m_i, \label{e:def_Z}\\
	\left\Vert\pi^m (D_i T(x)- D_i T(\bar{x})) \right\Vert_{B(X^i,X^m)} &\leq \sum_{j=1}^M W^m_{ij} \left\Vert \pi^j(x-\bar{x}) \right\Vert_{X^j}
	&\quad&\text{for all } x\in \bbox(\bar{x},\rstar). \label{e:def_W}
\end{alignat}
If $r,\eta \in \RR_{>0}^M$ with $r \leq \rstar$ satisfy
\begin{alignat}{1}
	Y^m + \sum_{i=1}^M Z^m_{i} r^i  
	+ \frac{1}{2}\sum_{i,j=1}^M W^m_{ij} r^i r^j  & \leq r^m 
	\label{e:inequalities1} \\
	 \sum_{i=1}^M Z^m_{i} \eta^i
	 + \sum_{i,j=1}^M W^m_{ij} \eta^i r^j  & < \eta^m,
	 \label{e:inequalities2} 
\end{alignat}
for $1\leq m\leq M$,
then $T$ has a unique fixed point in $\bbox(\bar{x},r)$.
\end{theorem}
\begin{remark}
In the case $M=1$, for which we can drop the exponents $m$ and the indices $i$ and~$j$, we recover a statement which is often used in the computer-assisted proof literature. Indeed, $Y$, $Z$ and $W$ are then simply real numbers which have to satisfy
 \begin{alignat}{2}
	\left\Vert T(\bar{x})-\bar{x} \right\Vert_{X} &\leq Y, \label{e:def_Y_M1}\\
	\left\Vert D T(\bar{x}) \right\Vert_{B(X,X)} &\leq Z, \label{e:def_Z_M1}\\
	\left\Vert D T(x)- D T(\bar{x}) \right\Vert_{B(X,X)} &\leq  W \left\Vert x-\bar{x} \right\Vert_{X}
	&\quad&\text{for all } x\in \B_X(\bar{x},\rstar), \label{e:def_W_M1}
\end{alignat}
where $\B_X(\bar{x},\rstar)$ is the ball of center $\bar{x}$ and radius $\rstar$ in $X$. Moreover, it is then sufficient to have the existence of $r\in(0,\rstar]$ satisfying 
\begin{alignat}{1}
	Y + Z r	+ \frac{1}{2}W r^2  & \leq r \label{e:inequalities1_M1}
	\\
	 Z + W r  & < 1,\label{e:inequalities2_M1}
\end{alignat}
in order to conclude that $T$ has a unique fixed point in $\B_X(\bar{x},r)$. This simplified version will be used in Sections~\ref{sec:onedomain} and~\ref{sec:infinity}.
\end{remark}

%!TEX root = main.tex
\section{Taking $\L$ constant over the entire integration time}
\label{sec:onedomain}

In this section, we define explicitly a suitable time-independent approximation $\L$ of $DF(\bu(t))$, consider the corresponding fixed point operator $T$ introduced in~\eqref{eq:T_general}, and then derive computable bounds $Y$, $Z$ and $W$ satisfying assumptions~\eqref{e:def_Y_M1}-\eqref{e:def_W_M1}.

We hasten to emphasize that this section is mostly here for pedagogical purposes. Indeed, as soon as the integration time $\tau$ is not very small, in practice we always take an operator $\L(t)$ which is only piece-wise constant and not fully time-independent. However, first considering $\L$ time-independent allows us to present the derivation of the $Y$, $Z$ and $W$ bounds with one less layer of notation, and all the computations made in this section will in fact be useful in Section~\ref{sec:domaindecomposition}, when we finally take $\L(t)$ piece-wise constant. Also, the considerations made in this section on a single time-domain will also prove useful again in Section~\ref{sec:infinity}, when we consider an infinitely long step near a stable equilibrium.

In this entire section, $\bu$ denotes an approximate solution to~\eqref{eq:PDE}, which is assumed to belong to $\Pi^{\leq \Nu} \X_\nu$ for some $\Nu\in\NN$. In practice, we use polynomials in time to construct our approximate solution, i.e., each Fourier coefficient $\bu_n\in C^0{([0,\tau],\CC)}$ is a polynomial.

\subsection{Construction of $\L$}
\label{sec:L_onedomain}

As in Section~\ref{sec:introduction}, we denote by $F$ the right-hand-side of~\eqref{eq:PDE}, written in Fourier space. That is, with the notation introduced in Section~\ref{sec:preliminaries_Fourier} and Section~\ref{sec:preliminaries_Xnu}, we consider the map $F$ defined on $\X_\nu$ by
\begin{align*}
F(u) = (-1)^{J+1} \D^{2J}u + \displaystyle\sum_{j=0}^{2J-1} \D^{j}g^{(j)}(u).
\end{align*}
Remember that we want $\L$ to be an approximation of $DF(\bu)$, where, for $h\in\X_\nu$,
\begin{align*}
DF(\bu)h = (-1)^{J+1} \D^{2J}h + \displaystyle\sum_{j=0}^{2J-1} \D^{j}\left[\left( \left(g^{(j)}\right)' (\bu)\right)\ast h\right].
\end{align*}
Since we first want $\L$ to be time independent, we compute finite vectors $\bv^{(j)} \in\Pi^{\leq \Nu}\ell^1_\nu$, $j=1,\ldots,2J-1$, such that
\begin{align*}
\bv^{(j)} \approx \frac{1}{\tau}\int_{0}^{\tau} \Pi^{\leq \Nu}\left[\left(g^{(j)}\right)'(\bu(t))\right] \d t.
\end{align*}
\begin{remark}
In practice, we numerically compute $\frac{1}{\tau}\int_{0}^{\tau} \Pi^{\leq \Nu}\left[\left(g^{(j)}\right)'(\bu(t))\right] \d t$ and take for $\bv^{(j)}$ the output of this computation, but all the estimates to come would be valid for any choice of $\bv^{(j)} \in\Pi^{\leq \Nu}\ell^1_\nu$. In particular, the $\approx$ sign above is only here to emphasize that this computation giving the $\bv^{(j)}$'s does not have to be done rigorously. 
\end{remark}

Using these $\bv^{(j)}$, we have a natural time-independent approximation of $DF(\bu)$, given by
\begin{align*}
h \mapsto (-1)^{J+1} \D^{2J}h + \displaystyle\sum_{j=0}^{2J-1} \D^{j}\left(\bv^{(j)}\ast h\right).
\end{align*}
We further simplify this linear operator, by only keeping diagonal terms in this approximation for large enough modes.
In order to make this precise, let us fix a threshold $\NL\in\NN$, and consider
\begin{align}
\label{eq:deflambda}
\lambda_n \bydef -n^{2J} + \sum_{j=0}^{2J-1} (in)^j \bv^{(j)}_0,\qquad \text{for }\vert n\vert > \NL,
\end{align}
which are going to be the diagonal terms. We then define the (unbounded) operator $\tL$ acting on $\X_\nu$ by
\begin{align*}
\begin{cases}
\tL \,\Pi^{\leq \NL}u = (-1)^{J+1} \D^{2J} \Pi^{\leq \NL}u + \ds \sum_{j=0}^{2J-1} \D^{j} \Pi^{\leq \NL}\left( \bv^{(j)} \ast \Pi^{\leq \NL}u \right) \\
\left(\tL \,\Pi^{> \NL}u\right)_{n} = 
\begin{cases}
	0 & \quad \text{for } \vert n\vert \leq \NL\\
	\lambda_n u_n &\quad \text{for } \vert n\vert > \NL.
\end{cases}
\end{cases}
\end{align*}
\begin{remark}
Note that $\NL$ need not be equal to $\Nu$ (which dictates the number of Fourier modes used for the approximate solution $\bu$), and in practice we often take $\NL$ significantly smaller than~$\Nu$.
\end{remark}
\begin{remark}
An equivalent definition of $\tL$, which will prove convenient later on, is
\begin{align}
\label{eq:deftL}
\tL u = (-1)^{J+1} \D^{2J}u + \ds \sum_{j=0}^{2J-1} \D^{j} \Pi^{\leq \NL}\left( \bv^{(j)} \ast \Pi^{\leq \NL}u \right) + \sum_{j=0}^{2J-1} \bv^{(j)}_0 \D^{j} \Pi^{> \NL}u .
\end{align}
\end{remark}
Notice that
\begin{align*}
\tL = \Pi^{\leq \NL} \tL \,\Pi^{\leq \NL} + \Pi^{> \NL} \tL \,\Pi^{> \NL},
\end{align*}
where $\Pi^{> \NL} \tL \,\Pi^{> \NL}$ is diagonal, therefore computing $e^{t\tL}$ only requires computing the exponential of a finite matrix. In principle, that could be done rigorously by rigorously diagonalizing $\Pi^{\leq \NL} \tL \,\Pi^{\leq \NL}$, but we can also get away with only doing that approximately. That is, we consider a diagonal matrix $\Lambda_{\NL}= \text{diag}\left(\lambda_{-\NL},\ldots,\lambda_{\NL}\right)$ and an invertible matrix $Q_{\NL}$ such that
\begin{equation*}
\Pi^{\leq \NL} \tilde{\L} \, \Pi^{\leq \NL} \approx Q_{\NL} \Lambda_{\NL} Q_{\NL}^{-1},
\end{equation*}
and define
\begin{equation*}
\L_{\NL} \bydef Q_{\NL} \Lambda_{\NL} Q_{\NL}^{-1},
\end{equation*}
which represents a linear operator on $\Pi^{\leq \NL}\ell^1_\nu$, and which we naturally interpret as a linear operator on the full space $\ell^1_\nu$, and as a constant linear operator on $\X_\nu$.  We finally define 
\begin{align*}
\L \bydef \L_{\NL} + \Pi^{> \NL} \tL \,\Pi^{> \NL}.
\end{align*}
In other words, for all $u\in\X_\nu$,
\begin{equation*}
\left\{
\begin{aligned}
& \L  \,\Pi^{\leq \NL}u =  \L_{\NL}\Pi^{\leq \NL} u \\
&\left(\L \,\Pi^{> \NL}u\right)_n =  
\begin{cases}
	0 & \quad \text{for } \vert n\vert \leq \NL\\
	\lambda_n u_n &\quad \text{for } \vert n\vert > \NL.
\end{cases}
\end{aligned}
\right.
\end{equation*}

Before going further, note that $Q_{\NL}$ can also be seen as a constant operator on $\Pi^{\leq \NL}\X_\nu$. It will be helpful to consider the operator $Q$ defined on the whole space $\X_\nu$ as $Q \bydef Q_{\NL} + \Pi^{>\NL}$, i.e.,
\begin{equation*}
\left\{
\begin{aligned}
&Q \,\Pi^{\leq \NL} u =  Q_{\NL}\Pi^{\leq \NL} u \\
&Q \,\Pi^{>\NL} u = \Pi^{>\NL} u.
\end{aligned}
\right.
\end{equation*}
Then, we can write
\begin{align*}
\L = Q\Lambda Q^{-1},
\end{align*}
where $\Lambda$ is the infinite diagonal matrix having diagonal entries $\lambda_n$ for all $n\in\ZZ$. We recall that these $\lambda_n$ are defined explicitly in~\eqref{eq:deflambda} for $\vert n\vert> \NL$, and that the remaining ones for $\vert n\vert \leq  \NL$ are chosen when we fix $\Lambda_{\NL}$ (in practice, we take these to be numerically computed eigenvalues of $\Pi^{\leq \NL} \tL \,\Pi^{\leq \NL}$, see Remark~\ref{rem:diagonalization} for more details).

Note that, for any $t\in\RR$ we have
\begin{align*}
e^{t\L} = Q\, e^{t\Lambda}\, Q^{-1},
\end{align*}
where $\Lambda$ is a diagonal matrix whose entries are known explicitly, and we also know $Q$ and $Q^{-1}$ explicitly.

\medskip

Now that $\L$ is well defined, we are ready to derive computable bounds $Y$, $Z$ and $W$ satisfying assumptions~\eqref{e:def_Y_M1}-\eqref{e:def_W_M1} of Theorem~\ref{thm:NewtonKantorovich} in the case $M=1$, for the operator $T$ of~\eqref{eq:T_general}, the space $X=\X_\nu$, and the approximate solution $\bar{x} = \bu$. 
Note that, since $\L$ is time independent here, so is $\gamma \bydef F - \L$, hence we write $\gamma(u(s))$ instead of $\gamma(s,u(s))$ within Section~\ref{sec:onedomain}. We now study the operator $T:\X_\nu\to\X_\nu$ given by
\begin{align}
\label{eq:T_L}
T(u)(t) = e^{t\L} \uin + \int_0^t e^{(t-s)\L} \gamma(u(s)) \d s,\qquad t\in[0,\tau],
\end{align}
which is what the general formula~\eqref{eq:T_general} becomes in this particular setting.

\begin{remark}\label{rem:diagonalization}
We may run into difficulties when trying to approximately diagonalize $\Pi^{\leq \NL} \tL \,\Pi^{\leq \NL}$ in cases with eigenvalues of higher multiplicity or when eigenvalues are very close to each other. The main issue appearing in the estimates is that when $Q_{\NL}$ is badly conditioned, its inverse $Q_{\NL}^{-1}$ is large (in operator norm). One way we mediate such unfavourable situations in the code is to diagonalize a small perturbation of $\Pi^{\leq \NL} \tL \,\Pi^{\leq \NL}$. This slightly increases the size of the defect $\Pi^{\leq \NL} \tL \,\Pi^{\leq \NL} - \L_{\NL}$, but it can reduce the norm of $Q_N^{-1}$, and the net result may be an improvement of the bounds. 

Additionally, especially when $\Pi^{\leq \NL} \tL \,\Pi^{\leq \NL}$ is not symmetric (and hence diagonalization is less well-conditioned generally), the size of $Q_N^{-1}$ can vary substantially among different subdomains and it can also depend sensitively on the location of the gridpoints that define the subdomains. We refer to Section~\ref{sec:parameters} for more details on how we remedy these adverse effects when optimizing the grid. 

Finally, one alternative, which will likely improve robustness but which we leave for future work, is to allow for some off-diagonal terms in $\Lambda_N$, such as those appearing in a Jordan normal form, for example, and simply push through the effects of the evaluation of $e^{t\Lambda_N}$ for such a nondiagonal matrix $\Lambda_N$ in the bounds.
\end{remark}

\begin{remark}\label{rem:symmetry}
When the problem~\eqref{eq:PDE} conserves even or odd symmetry, we restrict the analysis to the subspace spanned by either cosines or sines. In that case, the (near) diagonalization is also performed on this subspace. This is of particular interest when integrating to infinity, see Section~\ref{sec:infinity}, since working in the symmetric subspace lifts the translation invariance and removes the associated zero eigenvalue from the spectrum of the stationary state. On a technical level, in the code the change of basis $Q^{-1}$ also incorporate a translation from (two-sided) symmetric Fourier coefficients to (one-sided) sine or cosine coefficients, and vice versa for $Q$.	
\end{remark}

\subsection{Y-bound}
\label{sec:Ybounds}

In this subsection, we deal with the residual bound $Y$ satisfying assumption~\eqref{e:def_Y_M1} of Theorem~\ref{thm:NewtonKantorovich}. We assume that $\Nin \geq \NL$ and that there exists an explicit $\epsilon^{\mathrm{in}}\geq 0$ such that the initial data $\uin$ satisfies $\left\Vert \Pi^{>\Nin}\uin \right\Vert_{\ell^1_\nu} \leq \epsilon^{\mathrm{in}}$. That is, we allow for an initial data which has an infinite Fourier series, but require an explicit bound on the tail. 

We then split the computation of the $Y$ bound as follows:
\begin{align*}
\left(T(\bu)-\bu\right)(t) &= e^{ t \L} \uin +  \int_{0}^{t} e^{(t-s)\L } 
		\gamma(\bu(s)) \mbox{d} s - \bu \left( t \right) \\
&= \left(e^{ t \L} \Pi^{\leq \Nin}\uin +  \int_{0}^{t} e^{(t-s)\L } 
		\gamma(\bu(s)) \mbox{d} s - \bu \left( t \right)\right) +  e^{ t \L} \Pi^{> \Nin}\uin.
\end{align*}
Recalling that $\left\Vert \cdot \right\Vert_{\X_\nu} = \left\Vert \left\Vert \cdot \right\Vert_{C^0} \right\Vert_{\ell^1_\nu}$, note that 
\begin{align*}
&\left\Vert t\mapsto e^{ t \L} \Pi^{\leq \Nin}\uin +  \int_{0}^{t} e^{(t-s)\L } 
		\gamma(\bu(s)) \mbox{d} s - \bu \left( t \right) \right\Vert_{\X_\nu}  \\
		&\qquad\qquad\qquad= \left\Vert \left\Vert t\mapsto Q\left(e^{ t \Lambda} Q^{-1}\Pi^{\leq \Nin}\uin +  \int_{0}^{t} e^{(t-s)\Lambda } 
		Q^{-1}\gamma(\bu(s)) \mbox{d} s \right) - \bu \left( t \right) \right\Vert_{C^0} \right\Vert_{\ell^1_\nu}
\end{align*}
can be computed almost exactly, meaning that a tight and guaranteed enclosure can be obtained. Indeed, this calculation only involves finitely many Fourier modes, since $\bu$ has only finitely many Fourier modes and all the nonlinear terms in $\gamma$ are polynomials. We then use the procedure described in Section~\ref{sec:interp} to get an explicit and sharp upper-bound of the $C^0$ norm for each of these modes.

\begin{remark}
\label{rem:splitting}
Let us denote the map appearing in the above estimate by $\varphi$, i.e.
\begin{align*}
\varphi(t) \bydef  e^{ t \L} \Pi^{\leq \Nin}\uin +  \int_{0}^{t} e^{(t-s)\L } 
		\gamma(\bu(s)) \mbox{d} s - \bu \left( t \right).
\end{align*}
According to Section~\ref{sec:interp}, we want to estimate $\left\Vert \varphi \right\Vert_{C^0}$ by $\left\Vert P_{\tilde K}\varphi \right\Vert_{C^0} + \left\Vert (I-P_{\tilde K})\varphi \right\Vert_{C^0}$. However, directly controlling the interpolation error $(I-P_{\tilde K})\varphi$ can be somewhat costly,  as it requires rigorously and explicitly bounding $\varphi$ on an ellipse in the complex plane~\cite{Tre13} ($C^k$-type interpolation errors are also costly and give worst bounds). This task becomes much easier if one accepts to first split the error estimate as
\begin{align}
\label{eq:splitinterperror}
\left\Vert (I-P_{\tilde K})\varphi \right\Vert_{C^0} \leq \left\Vert (I-P_{\tilde K})\varphi_{\mathrm{exp}} \right\Vert_{C^0} + \left\Vert (I-P_{\tilde K})\varphi_{\mathrm{int}} \right\Vert_{C^0},
\end{align}
with
\begin{align}
\label{eq:splittingold}
\varphi_{\mathrm{exp}}(t) =   e^{ t \L} \Pi^{\leq \Nin}\uin \qquad\text{and}\qquad \varphi_{\mathrm{int}}(t) = \int_{0}^{t} e^{(t-s)\L } 
		\gamma(\bu(s)) \mbox{d} s.
\end{align}
We note that there is no interpolation error for $\bu$ as soon as $\tilde{K}$ is larger than the degree in time used for defining $\bu$, which is always the case in practice. After having made this splitting, it becomes much easier to get a computable estimate for each term, see~\cite[Lemma 4.5 and Lemma 4.8]{BerBreShe24}, and this works well in practice.

However, with this splitting both $\varphi_{\mathrm{exp}}$ and $\varphi_{\mathrm{int}}$ are potentially large, even if $\varphi$ itself should be small (assuming $\bu$ is an accurate approximate solution), which means one might need a relatively large $\tilde{K}$ to get the right hand side of~\eqref{eq:splitinterperror} small. In order to alleviate this drawback, we can look for a different splitting for which $\varphi_{\mathrm{int}}$ would be smaller, because we can then expect $\left\Vert (I-P_{\tilde K})\varphi_{\mathrm{int}} \right\Vert_{C^0}$ to be small while using a smaller $\tilde{K}$. Note that if we achieve this, then the other interpolation error $\left\Vert (I-P_{\tilde K})\varphi_{\mathrm{exp}} \right\Vert_{C^0}$ must also be small (provided $\left\Vert (I-P_{\tilde K})\varphi \right\Vert_{C^0}$ itself was indeed small). In order to define this better splitting, let us first introduce the time-average $\bU \in\ell^1_\nu$ of $\bu$:
\begin{align*}
\bU = \frac{1}{\tau}\int_0^\tau \bu(s) \mbox{d} s,
\end{align*}
and then rewrite 
\begin{align*}
\varphi(t) &= e^{ t \L} \Pi^{\leq \Nin}\uin +  \int_{0}^{t} e^{(t-s)\L } 
		\gamma(\bU) \mbox{d} s + \int_{0}^{t} e^{(t-s)\L } \left(\gamma(\bu(s))
		-\gamma(\bU)\right) \mbox{d} s - \bu \left( t \right) \\
		&= e^{ t \L} \Pi^{\leq \Nin}\uin + \L^{-1}\left(e^{t\L}-I\right)\gamma(\bU) + \int_{0}^{t} e^{(t-s)\L } \left(\gamma(\bu(s))
		-\gamma(\bU)\right) \mbox{d} s - \bu \left( t \right).
\end{align*}
Note that
\begin{align*}
\L^{-1}\left(e^{t\L}-I\right) = Q \Lambda^{-1}\left(e^{t\Lambda}-I\right)Q^{-1},
\end{align*}
where $\Lambda$ is a diagonal matrix with entries $\lambda_n$, and where $\lambda_n^{-1}\left(e^{t\lambda_n}-1\right)$ must be replaced by $t$ if $\lambda_n=0$. We can then repeat the above splitting of the interpolation error estimate~\eqref{eq:splitinterperror}, but this time with
\begin{align}
\label{eq:splittingnew}
\varphi_{\mathrm{exp}}(t) =   e^{ t \L} \Pi^{\leq \Nin}\uin + \L^{-1}\left(e^{t\L}-I\right)\gamma(\bU) \quad\text{and}\quad \varphi_{\mathrm{int}}(t) = \int_{0}^{t} e^{(t-s)\L } 
		\left(\gamma(\bu(s))
		-\gamma(\bU)\right) \mbox{d} s.
\end{align}
If $\bu(t)$ remains somewhat close to $\bU$ for all $t\in[0,\tau]$, then we can expect this new splitting~\eqref{eq:splittingnew} to perform better than~\eqref{eq:splittingold}, meaning that we can retain a small interpolation error bound while using a smaller value of $\tilde{K}$. This new splitting is what we actually use in practice when implementing the bounds.
\end{remark}

Regarding the potential error term coming from the initial data, using that $\Nin\geq \NL$ we get
\begin{align*}
\left\Vert e^{ t \L} \Pi^{> \Nin}\uin \right\Vert_{\X_\nu} &\leq \left\Vert \left\Vert Qe^{ t \Lambda} Q^{-1}\Pi^{> \Nin}\right\Vert_{C^0} \right\Vert_{\ell^1_\nu} \epsilon^{\mathrm{in}} \\
&= \left\Vert \left\Vert e^{ t \Lambda} \Pi^{> \Nin}\right\Vert_{C^0} \right\Vert_{\ell^1_\nu} \epsilon^{\mathrm{in}} \\
&= \sup_{\vert n\vert > \Nin} \max_{t\in[0,\tau]} \left\vert e^{t\lambda_n}\right\vert \epsilon^{\mathrm{in}} .
\end{align*}
Introducing, for all $N\in\NN$,
\begin{align}
\label{eq:mu}
\mu_N \bydef \exp \left( \tau \sup_{\vert n\vert > N} \Re\left(\lambda_n\right)^+ \right),
\end{align}
where $x^+ \bydef \max(x,0)$, we can therefore consider
\begin{align*}
Y \bydef \left\Vert t\mapsto e^{ t \L} \Pi^{\leq \Nin}\uin +  \int_{0}^{t} e^{(t-s)\L } 
		\gamma(\bu(s)) \mbox{d} s - \bu \left( t \right) \right\Vert_{\X_\nu} + \mu_{\Nin} \epsilon^{\mathrm{in}},
\end{align*}
which is computable, or rather we can get a sharp upper bound on it (as explained in Remark~\ref{rem:splitting}), and satisfies assumption~\eqref{e:def_Y_M1} of Theorem~\ref{thm:NewtonKantorovich}.

\subsection{Z-bounds}
\label{sec:Zbounds}

In this subsection, we now focus on the bound $Z$ satisfying assumption~\eqref{e:def_Z_M1} of Theorem~\ref{thm:NewtonKantorovich}. We first introduce notations for various quantities that are going to appear in the derivation of this bound $Z$.
\begin{definition}[Notation for the $Z$-bound]
\label{def:notation_Z}
${}$
\begin{itemize}
\item For all $n\in\ZZ$, $\tb_n \bydef \dfrac{e^{\ds\tau \Re(\lambda_n)}-1}{\Re(\lambda_n)}$, with the convention that $\tb_n \bydef \tau$ if $\Re(\lambda_n) = 0$.
\item $B \bydef \vert Q\vert\, \Re(\Lambda)^{-1} \left(e^{\tau \Re(\Lambda)} - I\right) = \vert Q\vert \diag((\tb_n)_{n\in\ZZ})$. 
\item For all $j=0,\ldots,2J-1$, $\chi^{(j)}_{\NL} \bydef \sup\limits_{\vert n\vert >\NL} \vert n\vert^j\, \tilde{b}_n$. Note that $\chi^{(j)}_{\NL}$ is computable (see, e.g., \cite[Appendix D]{BerBreShe24}). 
\item For all $j=0,\ldots,2J-1$, $\beta^{(j)} \bydef \left\Vert\left(g^{(j)}\right)'(\bu) - \bar{v}_0^{(j)} \right\Vert_{C^0} \in \ell^1_\nu$.
\item Let $\bN\in\NN$, such that for all $\vert n\vert >\bN$ and all $j=0,\ldots,2J-1$, $\beta^{(j)}_n = 0$. Let $\Gamma$ be a linear operator such that, for all $a\in \ell^1_\nu$
\begin{align*}
\Gamma\, \Pi^{\leq \bN + \NL} \vert a\vert \geq \left\Vert Q^{-1} D\gamma(\bu) \right\Vert_{C^0} \Pi^{\leq \bN + \NL} \vert a\vert,
\end{align*}
and
\begin{align*}
\Gamma\, \Pi^{> \bN + \NL}a &\bydef  \sum_{j=0}^{2J-1} \vert \D^j\vert \left( \beta^{(j)} \ast \Pi^{> \bN +\NL} a \right).
\end{align*}
The operator $Q^{-1} D\gamma(\bu) \Pi^{\leq \bN + \NL}$ can be seen as a finite matrix whose coefficients are smooth functions of time, therefore $\left\Vert Q^{-1} D\gamma(\bu) \right\Vert_{C^0} \Pi^{\leq \bN + \NL}$ is a finite matrix of real numbers, for which element-wise sharp upper-bounds can be obtained following Section~\ref{sec:interp}. In practice,  we define $\Gamma\, \Pi^{\leq \bN + \NL}$ with these upper-bounds.
\end{itemize}
\end{definition}

We are now ready to derive the $Z$ bound. We start by reducing  this bound to controling an $\ell^1_\nu$ operator norm.
\begin{lemma}
\label{lem:fromZtoell1}
With the notations of Definition~\ref{def:notation_Z},
\begin{align*}
\left\Vert DT(\bu) \right\Vert_{B(\X_\nu,\X_\nu)} \leq \left\Vert B\Gamma \right\Vert_{B(\ell^1_\nu,\ell^1_\nu)}.
\end{align*}
\end{lemma}
\begin{proof}
We consider an abitrary $h$ in $\X_\nu$, and estimate
\begin{align*}
[DT(\bu)h] (t) &= \int_{0}^t e^{(t-s)\L} D\gamma(\bu(s))h(s) \d s \\
&= Q \int_{0}^t e^{(t-s)\Lambda} Q^{-1}D\gamma(\bu(s))h(s) \d s.
\end{align*}
Taking the $C^0$ norm in time component-wise, we get
\begin{align*}
\left\Vert DT(\bu)h \right\Vert_{C^0} &\leq \vert Q\vert  \int_{0}^t e^{(t-s)\Re(\Lambda)} \d s \left\Vert Q^{-1}D\gamma(\bu)h\right\Vert_{C^0} \\
&\leq B  \left\Vert Q^{-1}D\gamma(\bu)\right\Vert_{C^0}\left\Vert h\right\Vert_{C^0}.
\end{align*}
We now check that
\begin{align*}
\left\Vert Q^{-1}D\gamma(\bu)\right\Vert_{C^0}\left\Vert h\right\Vert_{C^0} \leq \Gamma \left\Vert h\right\Vert_{C^0}.
\end{align*}
According to the definition of $\Gamma$, this is automatic for the modes $\vert n\vert \leq \NL + \bN$, and we only have to prove that
\begin{align*}
\left\Vert Q^{-1}D\gamma(\bu)\right\Vert_{C^0}\, \Pi^{> \bN + \NL} \left\Vert h\right\Vert_{C^0} \leq \Gamma \, \Pi^{> \bN + \NL} \left\Vert h\right\Vert_{C^0}.
\end{align*}
By using~\eqref{eq:deftL} we obtain
\begin{align*}
D\gamma(\bu) \Pi^{> \bN + \NL}h  &= DF(\bu)\Pi^{> \bN + \NL} h - \L \Pi^{> \bN + \NL} h \\
 &= DF(\bu)\Pi^{> \bN + \NL}h - \tL \Pi^{> \bN + \NL}h \\
&= \sum_{j=0}^{2J-1} \D^j \left(\left[\left(g^{(j)}\right)'(\bu)\right]\ast \Pi^{> \bN + \NL} h \right) -  \sum_{j=0}^{2J-1}  \D^j \left(\bv^{(j)}_0\Pi^{> \bN + \NL} h\right)  \\
 &=  \sum_{j=0}^{2J-1} \D^j \left(\left[\left(g^{(j)}\right)'(\bu) - \bar{v}^{(j)}_0 \right]\ast \Pi^{> \bN + \NL} h \right),
\end{align*}
hence
\begin{align*}
\left\Vert Q^{-1}D\gamma(\bu)\right\Vert_{C^0}\, \Pi^{> \bN + \NL} \left\Vert h\right\Vert_{C^0} & \leq \vert Q^{-1}\vert \sum_{j=0}^{2J-1} \vert\D^j\vert \left(\left\Vert\left(g^{(j)}\right)'(\bu) - \bar{v}^{(j)}_0 \right\Vert_{C^0} \ast \Pi^{> \bN + \NL} \left\Vert h\right\Vert_{C^0} \right) \\
& = \vert Q^{-1}\vert \sum_{j=0}^{2J-1} \vert\D^j\vert \left(\beta^{(j)} \ast \Pi^{> \bN + \NL} \left\Vert h\right\Vert_{C^0} \right).
\end{align*}
Since $\beta^{(j)} \in\Pi^{\leq \bN}\ell^1_\nu$ for all $j$, we get
\begin{align*}
\beta^{(j)} \ast \Pi^{> \bN + \NL} \left\Vert h\right\Vert_{C^0} = \Pi^{>\NL} \left(\beta^{(j)} \ast \Pi^{> \bN + \NL} \left\Vert h\right\Vert_{C^0}\right),
\end{align*}
and noting that $\vert Q^{-1}\vert \Pi^{>\NL} = \Pi^{>\NL}$, we infer that
\begin{align*}
\left\Vert Q^{-1}D\gamma(\bu)\right\Vert_{C^0}\, \Pi^{> \bN + \NL} \left\Vert h\right\Vert_{C^0} & \leq \sum_{j=0}^{2J-1} \vert\D^j\vert \left(\beta^{(j)} \ast \Pi^{> \bN + \NL} \left\Vert h\right\Vert_{C^0} \right)\\
& = \Gamma\, \Pi^{> \bN + \NL} \left\Vert h\right\Vert_{C^0}.
\end{align*}
We have therefore proven that
\begin{align*}
\left\Vert DT(\bu)h \right\Vert_{C^0} \leq B\Gamma \left\Vert h \right\Vert_{C^0},
\end{align*}
and the lemma then follows directly from the fact that $\left\Vert \cdot \right\Vert_{\X_\nu} = \left\Vert \left\Vert \cdot \right\Vert_{C^0} \right\Vert_{\ell^1_\nu}$.
\end{proof}
We now provide a computable upper-bound for $\left\Vert B\Gamma \right\Vert_{B(\ell^1_\nu,\ell^1_\nu)}$.
\begin{proposition}
\label{prop:ell1forZ}
We have
\begin{align*}
\left\Vert B\Gamma \right\Vert_{B(\ell^1_\nu,\ell^1_\nu)} \leq \max\left( \left\Vert B\Gamma \, \Pi^{\leq \bN+\NL} \right\Vert_{B(\ell^1_\nu,\ell^1_\nu)}, \sum_{j=0}^{2J-1} \left\Vert \beta^{(j)}\right\Vert_{\ell^1_\nu} \chi^{(j)}_{\NL}\right).
\end{align*}
\end{proposition}
\begin{proof}
From the properties of the $\ell^1$ operator norm, we have
\begin{align*}
\left\Vert B\Gamma \right\Vert_{B(\ell^1_\nu,\ell^1_\nu)} = \max\left( \left\Vert B\Gamma \, \Pi^{\leq \bN+\NL} \right\Vert_{B(\ell^1_\nu,\ell^1_\nu)}, \left\Vert B\Gamma \, \Pi^{> \bN+\NL} \right\Vert_{B(\ell^1_\nu,\ell^1_\nu)}\right),
\end{align*}
and we simply need to estimate the second term. We recall that $\Gamma \, \Pi^{> \bN+\NL} = \Pi^{> \NL}\, \Gamma \, \Pi^{> \bN+\NL}$, and thus
\begin{align*}
\left\Vert B\Gamma \,\Pi^{> \bN+\NL} \right\Vert_{B(\ell^1_\nu,\ell^1_\nu)} &\leq \sum_{j=0}^{2J-1} \left\Vert \Pi^{> \NL} B \vert \D^{j}\vert \right\Vert_{B(\ell^1_\nu,\ell^1_\nu)} \left\Vert \beta^{(j)}  \right\Vert_{\ell^1_\nu}.
\end{align*}
Since $\Pi^{> \NL} B \vert \D^{j}\vert$ is a diagonal operator, we simply have
\begin{equation*}
\left\Vert \Pi^{> \NL} B \vert \D^{j}\vert \right\Vert_{B(\ell^1_\nu,\ell^1_\nu)} \leq \sup_{\vert n\vert >\NL} \tilde{b}_n \vert n\vert^j =  \chi^{(j)}_{\NL}. \qedhere
\end{equation*}
\end{proof}
Note that each column of $B\Gamma$ is finite, hence $\left\Vert B\Gamma \, \Pi^{\leq \bN+\NL} \right\Vert_{B(\ell^1_\nu,\ell^1_\nu)}$ really is computable. We therefore take
\begin{align*}
Z \bydef \max\left( \left\Vert B\Gamma \, \Pi^{\leq \bN+\NL} \right\Vert_{B(\ell^1_\nu,\ell^1_\nu)}, \sum_{j=0}^{2J-1} \left\Vert \beta^{(j)}\right\Vert_{\ell^1_\nu} \chi^{(j)}_{\NL}\right),
\end{align*}
which is computable and satisfies assumption~\eqref{e:def_Z_M1} of Theorem~\ref{thm:NewtonKantorovich}.

\subsection{W-bound}
\label{sec:Wbounds}

In this subsection, we finally derive a computable bound $W$ satisfying assumption~\eqref{e:def_W_M1} of Theorem~\ref{thm:NewtonKantorovich}. In order to estimate the norm of $DT(u) - DT(\bu)$, for any $u$ in $B_{\X_\nu}(\bu,\rstar)$, we write $v=u-\bu$ and compute, for any $h\in\X_\nu$,
\begin{align}
	& \left[ \left(DT \left( \bu + v \right) - DT \left( \bu \right) \right)h \right](t) \nonumber \\[1ex] 
	 & \qquad = 
	\int_{0}^{t} e^{ \left( t - s\right) \L } 
	\biggl( D\gamma \left( \bu\left(s\right) + v\left(s\right) \right) - D\gamma \left( \bu \left(s\right) \right) \biggr) h(s) \ \mbox{d}s
	\nonumber \\[1ex] 
	& \qquad = 	
 \int_{0}^{t} e^{ \left( t - s\right) \L} 
	 \int_{0}^{1} D^{2}\gamma \left( \bu \left(s\right) + \tilde{s} v\left(s\right) \right) 
		\left[ h(s), v(s) \right] \ \mbox{d} \tilde{s} \ \mbox{d}s
	\nonumber \\[1ex]
	& \qquad = 	
 \int_{0}^{t} e^{ \left( t - s\right) \L} \sum_{j=0}^{2J-1} \D^j
	 \int_{0}^{1}  \left[\left(g^{(j)}\right)'' \! \left( \bu \left(s\right) + \tilde{s} v\left(s\right)\right)\right] \ast v(s) \ast h(s)  \ \mbox{d} \tilde{s} \ \mbox{d}s .
	\label{eq:D2T}
\end{align}
Taking the $C^0$ norm component-wise gives, with $B$ defined in Definition~\ref{def:notation_Z},
\begin{align*}
\left\Vert \left(DT(\bu+ v) - DT(\bu)\right)h \right\Vert_{C^0} &\leq B\sum_{j=0}^{2J-1} \vert \D^j\vert  \left( \left \vert \left( g^{(j)} \right)'' \right \vert \left( \Vert \bar u \Vert_{C^{0}} +  \Vert v \Vert_{C^{0}} \right)  \ast \Vert v \Vert_{C^{0}} \ast \Vert h \Vert_{C^{0}}	\right),
\end{align*}
where $\left \vert \left( g^{(j)} \right)'' \right \vert$ is the polynomial defined as in Remark~\ref{rem:absg}. 
We then apply the $\ell^1_\nu$-norm and estimate, using Lemma~\ref{lem:abs_pol},
\begin{align}
\label{eq:W1}
&\left\Vert \left(DT(\bu+ v) - DT(\bu)\right)h \right\Vert_{\X_\nu} \nonumber\\
&\qquad = \left\Vert\left\Vert \left(DT(\bu+ v) - DT(\bu)\right)h \right\Vert_{C^0}\right\Vert_{\ell^1_\nu} \nonumber\\
&\qquad\leq\sum_{j=0}^{2J-1} \left\Vert B \, \vert\D^j\vert\right\Vert_{B(\ell^1_\nu,\ell^1_\nu)}   \left \vert \left( g^{(j)} \right)'' \right \vert \left( \left\Vert\left\Vert \bar u \right\Vert_{C^{0}}\right\Vert_{\ell^1_\nu} +  \left\Vert\left\Vert v \right\Vert_{C^{0}}\right\Vert_{\ell^1_\nu} \right)  \left\Vert\left\Vert v \right\Vert_{C^{0}}\right\Vert_{\ell^1_\nu} \left\Vert\left\Vert h \right\Vert_{C^{0}}\right\Vert_{\ell^1_\nu} \nonumber\\
&\qquad\leq  \sum_{j=0}^{2J-1} \left\Vert B \, \vert\D^j\vert\right\Vert_{B(\ell^1_\nu,\ell^1_\nu)} \left \vert \left( g^{(j)} \right)'' \right \vert \left( \left\Vert \bar u \right\Vert_{\X_\nu} + \rstar \right) \left\Vert v \right\Vert_{\X_\nu}  \left\Vert h \right\Vert_{\X_\nu}.
\end{align}
Notice that $\left\Vert B \, \vert\D^j\vert\right\Vert_{B(\ell^1_\nu,\ell^1_\nu)}$ can be computed explicitly (see the proof of Proposition~\ref{prop:ell1forZ}), as 
\begin{align*}
\left\Vert B \, \vert\D^j\vert\right\Vert_{B(\ell^1_\nu,\ell^1_\nu)} = \max\left(\left\Vert \Pi^{\leq \NL} B\, \vert \D^j\vert\, \Pi^{\leq \NL}\right\Vert_{B(\ell^1_\nu,\ell^1_\nu)} ,\, \chi^{(j)}_N\right) .
\end{align*}
Therefore
\begin{align*}
W = \sum_{j=0}^{2J-1} \left\Vert B \, \vert\D^j\vert\right\Vert_{B(\ell^1_\nu,\ell^1_\nu)}   \left \vert \left( g^{(j)} \right)'' \right \vert \left( \left\Vert \bar u \right\Vert_{\X_\nu} + \rstar \right)
\end{align*}
is computable and satisfies assumption~\eqref{e:def_W_M1} of Theorem~\ref{thm:NewtonKantorovich}.

%!TEX root = main.tex

\section{Longer integration times}
\label{sec:domaindecomposition}

In Section~\ref{sec:onedomain}, we have derived estimates $Y$, $Z$ and $W$ satifying the assumptions~\eqref{e:def_Y}-\eqref{e:def_W} of Theorem~\ref{thm:NewtonKantorovich}, for $T$ defined in~\eqref{eq:T_L}, assuming $\L$ was constant over the entire time interval $[0,\tau]$. However, since $\L$ was taken constant, we can only expect $T$ to be contracting if the actual solution does not vary too much over the entire time interval $[0,\tau]$, or equivalently, if $\tau$ is small enough. As explained in the introduction, a natural option to rigorously integrate over longer time intervals is to take $\L$ only piece-wise constant, hereby allowing $\L$ to better approximate $DF(\bu(t))$ over the whole interval $[0,\tau]$. We describe this strategy in detail in this section, and derive the corresponding $Y$, $Z$ and $W$ estimates.

\subsection{How to adapt the setup}
\label{sec:setup_DD}

Let $M\in\NN_{\geq 1}$ and consider a subdivision of the interval $[0,\tau]$ given by
\begin{align}
\label{eq:subdivision}
0 = \tau_0 < \tau_1 < \ldots < \tau_M = \tau.
\end{align}
For each $m=1,\ldots,M$, let $\L^{(m)}$ be a constant operator (whose precise definition will be given later on, and which generates a $C_0$ semigroup) and define
\begin{align}
\label{eq:L_piecewise}
\L(t) \bydef \L^{(m)},\quad t\in(\tau_{m-1},\tau_m].
\end{align}
In that case, the evolution operator $U$ associated to $\L$ is no longer a simple exponential, but it can still be expressed as a finite combination of exponentials. More precisely, if $0\leq s\leq t\leq \tau$ are such that $\tau_{m-1}\leq s\leq t\leq \tau_m$ for some $m$, then $U(t,s) = e^{(t-s)\L^{(m)}}$. Otherwise, let us introduce
\begin{align}
\label{eq:defkL}
\kL^{(m,l)} \bydef
\left\{
\begin{aligned}
& e^{(\tau_{m-1}-\tau_{m-2})\L^{(m-1)}} e^{(\tau_{m-2}-\tau_{m-3})\L^{(m-2)}} \ldots e^{(\tau_{l+1}-\tau_{l})\L^{(l+1)}} \qquad & 0\leq l\leq m-2, \\
& I , \qquad &l= m-1.
\end{aligned}
\right.
\end{align}
For $l$ and $m$ such that $\tau_{l-1} \leq s \leq \tau_l \leq \tau_{m-1} \leq t \leq \tau_{m}$,
\begin{align*}
U(t,s) = e^{(t-\tau_{m-1})\L^{(m)}} \kL^{(m,l)} e^{(\tau_l -s)\L^{(l)}}.
\end{align*}
Then, the operator $T$ given by~\eqref{eq:T_general} is still a well defined operator on $\X_\nu$ (recall Definition~\ref{def:Xnu}), and a fixed point of $T$ yields a solution of~\eqref{eq:PDE}.

In spirit, this is essentially the setup we consider to validate solutions on longer time intervals. However, in practice it will prove convenient to use instead a slightly different approach, based several copies of the $\X_\nu$ space corresponding to each subdomain of the above subdivision. We therefore consider, for each $m\in\{1,\ldots,M\}$,
\begin{align*}
\X_\nu^{(m)} \bydef \left\{ u^{(m)} = \left(u^{(m)}_n\right)_{n\in\ZZ}\in \left(C([\tau_{m-1},\tau_m],\CC)\right)^\ZZ,\ \left\Vert u^{(m)}\right\Vert_{\X_\nu^{(m)}} \bydef \sum_{n\in\ZZ} \left\Vert u^{(m)}_n\right\Vert_{C^0} \nu^{\vert n\vert} < \infty \right\},
\end{align*}
where the $C^0$ norm is on the interval where the function is defined, i.e.,
\begin{align*}
\left\Vert u^{(m)}_n\right\Vert_{C^0} \bydef \sup\limits_{t\in [\tau_{m-1},\tau_m]} \left\vert u^{(m)}_n(t)\right\vert,
\end{align*}
and the product space $\X_\nu^\pprod \bydef \ds \prod_{m=1}^M \X_\nu^{(m)}$. 
\begin{remark}
We have a norm on each space $\X_\nu^{(m)}$, but we purposely do not specify a norm on the product space $\X_\nu^\pprod$. Indeed, the main purpose of introducing this product space is that it provides us with the flexibility of taking different weights on each subintervals, which will be determined a posteriori when trying to apply the contraction mapping theorem, see Theorem~\ref{thm:NewtonKantorovich} and the discussion that precedes it.

Another, more minor, advantage of this product space, is that it allows us to take an approximate solution $\bu$ which is only piece-wise continuous but not necessarily globally continuous in time. However, we will design the fixed point operator corresponding to~\eqref{eq:PDE} in such a way that its fixed points are guaranteed to be continuous. 
\end{remark}
Henceforth, $\bu$ denotes an element of $\Pi^{\leq \Nu}\X_\nu^\pprod$, i.e., $\bu^{(m)}\in \Pi^{\leq \Nu}\X_\nu^{(m)}$ for all $m=1,\ldots,M$, which represent a piece-wise polynomial (in time) approximate solution to~\eqref{eq:PDE}. That is, we assume that each $u^{(m)}_n\in C^0([\tau_{m-1},\tau_m],\CC)$ is a polynomial, for $\vert n\vert\leq \Nu$ and $m=1,\ldots,M$.

\medskip

We are now ready to define, on each subinterval $[\tau_{m-1},\tau_m]$, a constant operator $\L^{(m)}$ approximating $DF(\bu^{(m)}(t))$, by mimicking the construction of Section~\ref{sec:L_onedomain}. That is, for each $m=1,\ldots,M$, we introduce:
\begin{itemize}
\item Finite vectors $\bv^{(m,j)} \in\Pi^{\leq \Nu}\ell^1_\nu$, $j=1,\ldots,2J-1$, such that
\begin{align*}
\bv^{(m,j)} \approx \frac{1}{\tau_m-\tau_{m-1}}\int_{\tau_{m-1}}^{\tau_m} \Pi^{\leq \Nu}\left[\left(g^{(j)}\right)'\left(\bu^{(m)}(t)\right)\right] \d t.
\end{align*}
\item An intermediate constant approximation $\tL^{(m)}$ of $DF(\bu^{(m)})$, given by
\begin{align*}
\tL^{(m)} a \bydef (-1)^{J+1} \D^{2J}a + \ds \sum_{j=0}^{2J-1} \D^{j}\Pi^{\leq \NL}\left( \bv^{(m,j)} \ast \Pi^{\leq \NL}a \right) + \sum_{j=0}^{2J-1} \bv^{(m,j)}_0 \D^{j} \Pi^{> \NL}a,
\end{align*}
for all $a\in\ell^1_\nu$.
\item A diagonal matrix $\Lambda_{\NL}^{(m)}= \text{diag}\left(\lambda_{-\NL}^{(m)},\ldots,\lambda_{\NL}^{(m)}\right)$ and an invertible matrix $Q_{\NL}^{(m)}$ such that
\begin{equation*}
\Pi^{\leq \NL} \tL^{(m)} \Pi^{\leq \NL} \approx Q_{\NL}^{(m)} \Lambda_{\NL} \left(Q_{\NL}^{(m)}\right)^{-1},
\end{equation*}
and
\begin{equation*}
\L_{\NL}^{(m)} \bydef Q_{\NL}^{(m)} \Lambda_{\NL}^{(m)} \left(Q_{\NL}^{(m)}\right)^{-1}.
\end{equation*}
\item The numbers 
\begin{align*}
\lambda_n^{(m)} \bydef -n^{2J} + \sum_{j=0}^{2J-1} (in)^j \,  \bv^{(m,j)}_0,\qquad \vert n\vert > \NL,
\end{align*}
the associated diagonal operator $\Lambda^{(m)}\bydef \diag\left(\left(\lambda^{(m)}_n\right)_{n\in\ZZ}\right)$, and the operator $Q^{(m)}\bydef Q_{\NL}^{(m)} \Pi^{\leq \NL} + \Pi^{> \NL}$.
\item The final constant approximation $\L^{(m)}$ of $DF(\bu^{(m)})$, given by
\begin{equation*}
\L^{(m)} \bydef Q^{(m)} \Lambda^{(m)} \left(Q^{(m)}\right)^{-1}.
\end{equation*}
\end{itemize}
Next, we introduce, for each $m=1,\ldots,M$, the time-independent function \begin{align*}
\gamma^{(m)}:u^{(m)} \mapsto F(u^{(m)}) - \L^{(m)}u^{(m)},
\end{align*}
and consider the operator $T^\pprod= \left(T^{(1)},\ldots,T^{(M)}\right)$ acting on $u=\left(u^{(1)},\ldots,u^{(M)}\right)\in \X_\nu^\pprod$ as
\begin{align}
\label{eq:defTprod}
&T^{(m)}(u)(t) \bydef e^{(t-\tau_{m-1})\L^{(m)}}\theta^{(m-1)}(u) + \int_{\tau_{m-1}}^t e^{(t-s)\L^{(m)}} \gamma^{(m)} \left(u^{(m)}(s)\right) \d s, 
\end{align}
for $t\in[\tau_{m-1},\tau_{m}]$ and $m=1,\ldots,M$, 
where the recursively defined
\begin{align}
\label{eq:deftheta}
% \begin{cases}
% \ds \theta^{(0)}(u) \bydef \uin,\\
\ds \theta^{(m)}(u) \bydef e^{(\tau_{m}-\tau_{m-1})\L^{(m)}}\theta^{(m-1)}(u) + \int_{\tau_{m-1}}^{\tau_m}e^{(\tau_{m}-s)\L^{(m)}} \gamma^{(m)} \left(u^{(m)}(s)\right) \d s,\quad m=1,\ldots,M-1,
% \end{cases}
\end{align}
with $\theta^{(0)}(u) \bydef \uin$, denotes the evaluation at the endpoint of the previous subdomain.
Up to the fact that the domain has been split into different pieces, this operator $T^\pprod$ is essentially the operator $T$ obtained from~\eqref{eq:T_general} by taking $\L$ piece-wise constant as in~\eqref{eq:L_piecewise}. In particular, note that if $u=\left(u^{(1)},\ldots,u^{(M)}\right)\in \X_\nu^\pprod$ is a fixed point of $T^\pprod$, then $u^{(m)}(\tau_m) = u^{(m+1)}(\tau_m)$ for all $m=1,\ldots,M-1$, therefore $u$ can also be interpreted as an element of $\X_\nu$, which is a fixed point of $T$ from~\eqref{eq:T_general} and hence a solution of~\eqref{eq:PDE}.

\medskip

We are now ready to derive $Y$, $Z$ and $W$ bounds satisfying the assumptions~\eqref{e:def_Y}-\eqref{e:def_W} of Theorem~\ref{thm:NewtonKantorovich}, for the operator $T^\pprod$, the product space $\X_\nu^\pprod$ and the approximate solution $\bu\in \X_\nu^\pprod$ introduced above. These bounds will at least partially be similar to the ones obtained in Section~\ref{sec:onedomain} when $\L$ was taken fully constant, but with an extra layer corresponding to the different subdomains.

\subsection{$Y$-bound}
\label{sec:Ybounds_DD}

In this subsection, we deal with the residual bounds $Y^{(m)}$ satisfying assumption~\eqref{e:def_Y} of Theorem~\ref{thm:NewtonKantorovich}. As in Section~\ref{sec:Ybounds}, we assume there exist $\Nin \geq \NL$ and an explicit $\epsilon^{\mathrm{in}}\geq 0$ such that $\left\Vert \Pi^{>\Nin}\uin \right\Vert_{\ell^1_\nu} \leq \epsilon^{\mathrm{in}}$.
In order to easily deal with this potentially infinite tail in the initial condition, we introduce
\begin{align*}
\begin{cases}
\ds \btheta^{(0)}(u) \bydef \Pi^{\leq\Nin}\uin,\\
\ds \btheta^{(m)}(u) \bydef e^{(\tau_{m}-\tau_{m-1})\L^{(m)}}\btheta^{(m-1)}(u) + \int_{\tau_{m-1}}^{\tau_m}e^{(\tau_{m}-s)\L^{(m)}} \gamma^{(m)} \left(u^{(m)}(s)\right) \d s,\qquad m\geq 1.
\end{cases}
\end{align*}
When applied to the approximate solution $\bu\in\Pi^{\leq \NL}\X_\nu^\pprod$, these $\btheta^{(m)}(\bu)$ are all computable, and, for all $m=1,\ldots,M$,
\begin{align*}
\theta^{(m)}(u) - \btheta^{(m)}(u) = \kL^{(m,0)}\Pi^{>\Nin}\uin.
\end{align*} 
Generalizing~\eqref{eq:mu}, we introduce for any $N\in\NN$, $m\in\{1,\ldots,M\}$ and $l\in\{0,\ldots,m-1\}$,
\begin{equation}
\label{eq:mumi}
\mu^{(m,l)}_N \bydef  \exp\left(\sup_{\vert n\vert >N} \left((\tau_m-\tau_{m-1}) \Re\left(\lambda_n^{(m)}\right)^+ + \sum_{j=l+1}^{m-1}(\tau_j-\tau_{j-1})\Re\left(\lambda_n^{(j)}\right)\right)\right),
\end{equation}
where $x^+ := \max(x,0)$, and with the convention that $\sum_{j=m}^{m-1}=0$. We then have, for all $m=1,\ldots,M$,
\begin{align*}
\left\Vert t\mapsto e^{(t-\tau_{m-1})\L^{(m)}}\kL^{(m,0)}\Pi^{>\Nin}\uin \right\Vert_{\X_\nu^{(m)}} \leq \mu_{\NL}^{(m,0)} \epsilon^{\mathrm{in}}.
\end{align*}

We then proceed exactly as in Section~\ref{sec:Ybounds}, and take, for all $m=1,\ldots,M$,
\begin{align*}
Y^{(m)} \!\bydef\! \left\Vert t\mapsto e^{(t-\tau_{m-1})\L^{(m)}} \btheta^{(m-1)}(\bu) +  \! \int_{\tau_{m-1}}^{t} \!\!\!\! e^{(t-s)\L^{(m)} } 
		\gamma^{(m)}\left(\bu^{(m)}(s)\right) \mbox{d} s - \bu^{(m)} \left( t \right) \right\Vert_{\X_\nu^{(m)}} \!\!\!\!+ \mu^{(m,0)}_{\Nin} \epsilon^{\mathrm{in}},
\end{align*}
which is computable and satisfies assumption~\eqref{e:def_Y} of Theorem~\ref{thm:NewtonKantorovich}.

\begin{remark}
Remark~\ref{rem:splitting} also applies here, and a splitting similar to~\eqref{eq:splittingnew} can be used on each subdomain, by rewriting 
\begin{align*}
&e^{(t-\tau_{m-1})\L^{(m)}} \btheta^{(m-1)}(\bu) +  \int_{\tau_{m-1}}^{t} e^{(t-s)\L^{(m)} } 
		\gamma^{(m)}\left(\bu^{(m)}(s)\right) \mbox{d} s - \bu^{(m)} \left( t \right)   \\
		&\qquad\qquad\qquad\qquad = e^{(t-\tau_{m-1})\L^{(m)}} \btheta^{(m-1)}(\bu) + (\L^{(m)})^{-1}\left(e^{(t-\tau_{m-1})\L^{(m)}}-I\right)\gamma^{(m)}(\bU^{(m)}) \\
		&\qquad\qquad \quad\qquad\qquad\qquad +  \int_{\tau_{m-1}}^{t} e^{(t-s)\L^{(m)} } 
		\left(\gamma^{(m)}\left(\bu^{(m)}(s)\right)-\gamma^{(m)}\left(\bU^{(m)}\right)\right) \mbox{d} s - \bu^{(m)} \left( t \right) ,
\end{align*}
with
$\bU^{(m)} = \frac{1}{\tau_m-\tau_{m-1}}\int_{\tau_{m-1}}^{\tau_m} \bu^{(m)}(s) \mbox{d} s$.
Note that the domain decomposition in time contributes to making this splitting more efficient, as the piece-wise constant approximation of $\bu^{(m)}$ by $\bU^{(m)}$ is going to be more accurate than a single global constant approximation over the entire interval~$[0,\tau]$.
\end{remark}

\subsection{$Z$-bound}
\label{sec:Zbounds_DD}

In this subsection, we deal with the bounds $Z^{(m)}_i$ satisfying assumption~\eqref{e:def_Z} of Theorem~\ref{thm:NewtonKantorovich}. 
To that end, let us first extend the notations introduced in Section~\ref{sec:Zbounds}.

\begin{definition}[Extension of Definition~\ref{def:notation_Z} to multiple subdomains]
\label{def:notation_Zmi}
For all $m\in\{1,\ldots,M\}$:
\begin{itemize}
\item For all $n\in\ZZ$, $\tb^{(m)}_n \bydef \dfrac{e^{(\tau_m-\tau_{m-1}) \Re(\lambda^{(m)}_n)}-1}{\Re(\lambda^{(m)}_n)}\in\RR$, with the convention that $\tb^{(m)}_n = \tau_m-\tau_{m-1}$ if $\Re(\lambda^{(m)}_n) = 0$.
\item For all $i\in\{1,\ldots,m\}$ and $\vert n\vert > \NL$,
\begin{align*}
b^{(m,i)}_n \!\bydef\! \left\{
\begin{aligned}
&\tb^{(m)}_n  && ~ i=m,\\
& \! \exp\biggl((\tau_m-\tau_{m-1}) \Re\left(\lambda_n^{(m)}\right)^+ + \sum_{j=i+1}^{m-1}(\tau_j-\tau_{j-1})\Re\left(\lambda_n^{(j)}\right)\biggr) \tb^{(i)}_n  && ~ i=1,\ldots,m-1,
\end{aligned}
\right.
\end{align*}
with the convention that $\sum_{j=m}^{m-1}=0$.
\item For all $i\in\{1,\ldots,m\}$, $B^{(m,i)}:\ell^1_\nu \to \ell^1_\nu$ is defined as follows
\begin{align*}
B^{(m,m)} \bydef \vert Q^{(m)}\vert\, \Re(\Lambda^{(m)})^{-1} \left(e^{(\tau_m-\tau_{m-1}) \Re(\Lambda^{(m)})} - I\right) ,
\end{align*} 
and, for $i\leq m-1$,
\begin{align*}
B^{(m,i)} &\bydef \vert Q^{(m)}\vert\, e^{(\tau_m-\tau_{m-1}) \Re(\Lambda^{(m)})^+} \left\vert \left(Q^{(m)}\right)^{-1} \kL^{(m,i)}  Q^{(i)} \right\vert 
 \Re(\Lambda^{(i)})^{-1} \left(e^{(\tau_i-\tau_{i-1}) \Re(\Lambda^{(i)})} - I\right) .
\end{align*}   
Note that
\begin{align*}
B^{(m,i)} = \Pi^{\leq \NL} B^{(m,i)} \Pi^{\leq \NL} + \Pi^{> \NL} B^{(m,i)} \Pi^{> \NL},
\end{align*}
where $\Pi^{> \NL} B^{(m,i)} \Pi^{> \NL}$ is diagonal and has diagonal elements given by $b^{(m,i)}_n$.
\item For all $j=0,\ldots,2J-1$, $\chi^{(m,j)}_{\NL} \bydef \sup\limits_{\vert n\vert >\NL} \vert n\vert^j\, \tilde{b}^{(m)}_n$. Note that $\chi^{(m,j)}_{\NL}$ is computable (see, e.g., \cite[Appendix D]{BerBreShe24}). 
\item For all $j=0,\ldots,2J-1$, $\beta^{(m,j)} \bydef \left\Vert\left(g^{(j)}\right)'(\bu^{(m)}) - \bar{v}_0^{(m,j)} \right\Vert_{C^0} \in \ell^1_\nu$.
\item Let $\bN\in\NN$, such that for all $\vert n\vert >\bN$, all $m=1,\ldots,M$ and all $j=0,\ldots,2J-1$, $\beta^{(m,j)}_n = 0$. Let $\Gamma^{(m)}$ be a linear operator such that, for all $a\in \ell^1_\nu$,
\begin{align*}
\Gamma^{(m)}\, \Pi^{\leq \bN + \NL} \vert a\vert \geq \left\Vert \left(Q^{(m)}\right)^{-1} D\gamma^{(m)}\left(\bu^{(m)}\right) \right\Vert_{C^0} \Pi^{\leq \bN + \NL} \vert a\vert,
\end{align*}
and
\begin{align*}
\Gamma^{(m)}\, \Pi^{> \bN + \NL}a &\bydef  \sum_{j=0}^{2J-1} \vert \D^j\vert \left( \beta^{(m,j)} \ast \Pi^{> \bN +\NL} a \right).
\end{align*}
\end{itemize}
\end{definition}
\begin{remark}
The operators $B^{(m,i)}$, which will prove critical in getting $Z$ bounds, depend on $\kL^{(m,i)}$. According to~\eqref{eq:defkL}, for $i\leq m-2$, $\kL^{(m,i)}$ is given by
\begin{align*}
 Q^{(m-1)}e^{(\tau_{m-1}-\tau_{m-2})\Lambda^{(m-1)}} \left(Q^{(m-1)}\right)^{-1}Q^{(m-2)}e^{(\tau_{m-2}-\tau_{m-3})\Lambda^{(m-2)}} \cdots e^{(\tau_{l+1}-\tau_{l})\Lambda^{(l+1)}}Q^{(l+1)}.
\end{align*}
In practice, especially when computing $\kL^{(m,i)}$ with interval arithmetic, it proves beneficial to first compute the $  
\left(Q^{(j)}\right)^{-1}Q^{(j-1)}$ products, because these should be somewhat close to the identity. 
\end{remark}

We are now ready to obtain $Z$ bounds, i.e., computable quantities $Z^{(m)}_i$ such that, for all $i,m\in\{1,\ldots,M\}$,
\begin{align*}
\left\Vert D_i T^{(m)}(\bu)\right\Vert_{B(\X_\nu^{(i)},\X_\nu^{(m)})} \leq Z^{(m)}_i.
\end{align*}
First note that $T^{\pprod}$ is lower triangular, i.e., $D_i T^{(m)}(\bu)=0$ for $i>m$, in which cases we simply take $Z^{(m)}_i=0$. When $i\leq m$, we proceed as in Section~\ref{sec:Zbounds}, and first reduce $\left\Vert D_i T^{(m)}(\bu)\right\Vert_{B(\X_\nu^{(i)},\X_\nu^{(m)})}$ to controling an $\ell^1_\nu$ operator norm.

\begin{lemma}
With the notations of Definition~\ref{def:notation_Zmi}, for all $1\leq i\leq m\leq M$,
\begin{align*}
\left\Vert D_iT^{(m)}(\bu) \right\Vert_{B(\X_\nu^{(i)},\X_\nu^{(m)})} \leq \left\Vert B^{(m,i)}\Gamma^{(i)} \right\Vert_{B(\ell^1_\nu,\ell^1_\nu)}.
\end{align*}
\end{lemma}
\begin{proof}
The proof if very similar to that of Lemma~\ref{lem:fromZtoell1}. We consider an arbitrary $h\in\X_\nu^{(i)}$ and estimate
\begin{align}
\label{eq:DmTm}
\left( D_m T^{(m)}(\bu)h \right)(t) &= \int_{\tau_{m-1}}^{t} e^{ (t-s)\L^{(m)} }  D\gamma^{(m)}(\bu^{(m)}(s)) (h(s)) \mbox{d} s, \qquad \text{if }i=m,
\end{align}
and 
\begin{align}
\label{eq:DiTm}
\left( D_i T^{(m)}(\bu)h \right)(t) &= e^{ (t-\tau_{m-1})\L^{(m)}}D_i\theta^{(m-1)}(\bu)h \nonumber\\
& = e^{ (t-\tau_{m-1})\L^{(m)}}\kL^{(m,i)}\int_{\tau_{i-1}}^{\tau_i} e^{(\tau_i-s)\L^{(i)} }  D\gamma^{(i)}(\bu^{(i)}(s)) (h(s)) \mbox{d} s, \qquad \text{if }i<m.
\end{align}
Taking the $C^0$-norm in time and proceeding exactly as in the $M=1$ case yields
\begin{align*}
\left\Vert D_iT^{(m)}(\bu)h \right\Vert_{C^0} &\leq B^{(m,i)}  \Gamma^{(i)} \left\Vert h\right\Vert_{C^0},
\end{align*}
and therefore
\begin{align*}
\left\Vert D_iT^{(m)}(\bu)h \right\Vert_{\X_\nu^{(m)}} &\leq \left\Vert B^{(m,i)} \Gamma^{(i)}\left\Vert h\right\Vert_{C^0} \right\Vert_{\ell^1_\nu} 
\leq \left\Vert B^{(m,i)} \Gamma^{(i)} \right\Vert_{B(\ell^1_\nu,\ell^1_\nu)} \left\Vert h \right\Vert_{\X_\nu^{(i)}}. \qedhere
\end{align*}
\end{proof}
We now provide a computable upper-bound for $\left\Vert B^{(m,i)} \Gamma^{(i)} \right\Vert_{B(\ell^1_\nu,\ell^1_\nu)}$.
\begin{proposition}
For all $1\leq i\leq m\leq M$, we have 
\begin{align*}
\left\Vert B^{(m,i)} \Gamma^{(i)} \right\Vert_{B(\ell^1_\nu,\ell^1_\nu)} \leq \max \left(\left\Vert B^{(m,i)} \Gamma^{(i)} \,\Pi^{\leq \bN+\NL}\right\Vert_{B(\ell^1_\nu,\ell^1_\nu)},\ \mu^{(m,i)}_{\NL}\sum_{j=0}^{2J-1} \left\Vert \beta^{(i,j)}\right\Vert_{\ell^1_\nu}  \chi^{(i,j)}_{\NL} \right),
\end{align*}
where $\mu^{(m,i)}_{\NL}$ is as in~\eqref{eq:mumi} if $i<m$, and $\mu^{(m,m)}_{\NL} \bydef 1$.
\end{proposition}
\begin{proof}
Again, the proof is very similar to that of Proposition~\ref{prop:ell1forZ}. In particular, we only need to estimate by hand
\begin{align*}
\left\Vert B^{(m,i)}\Gamma^{(i)} \right\Vert_{B(\ell^1_\nu,\ell^1_\nu)} &\leq \sum_{j=0}^{2J-1} \left\Vert \Pi^{> \NL} B^{(m,i)} \vert \D^{j}\vert \right\Vert_{B(\ell^1_\nu,\ell^1_\nu)} \left\Vert \beta^{(i,j)}  \right\Vert_{\ell^1_\nu},
\end{align*}
where
\begin{align*}
\left\Vert \Pi^{> \NL} B^{(m,i)} \vert \D^{j}\vert \right\Vert_{B(\ell^1_\nu,\ell^1_\nu)} &\leq \sup_{\vert n\vert >\NL} b^{m,i}_n\, \vert n\vert^j \\
&\leq \sup_{\vert n\vert >\NL} \mu^{(m,i)}_{\NL}\, \tilde{b}^{(i)}_n\, \vert n\vert^j \\
&\leq  \mu^{(m,i)}_{\NL}\, \chi^{(i,j)}_{\NL}. \qedhere
\end{align*}
\end{proof}
We therefore take
\begin{align*}
Z^{(m)}_i \bydef \max \left(\left\Vert B^{(m,i)} \Gamma^{(i)} \,\Pi^{\leq \bN+\NL}\right\Vert_{B(\ell^1_\nu,\ell^1_\nu)},\ \mu^{(m,i)}_{\NL}\sum_{j=0}^{2J-1} \left\Vert \beta^{(i,j)}\right\Vert_{\ell^1_\nu}  \chi^{(i,j)}_{\NL} \right),
\end{align*}
which is computable and satisfies assumption~\eqref{e:def_Z} of Theorem~\ref{thm:NewtonKantorovich}.

\subsection{$W$-bound}
\label{sec:Wbounds_DD}

In this subsection, we finally derive computable bounds $W^{(m)}_{i,j}$ satisfying assumption~\eqref{e:def_W} of Theorem~\ref{thm:NewtonKantorovich}. For any $u\in\X^\pprod_\nu$ and $r=(r^{(1)},\ldots,r^{(m)}) \in \RR^M_{>0}$, we define 
\begin{align*}
\bbox(u,r) = \prod_{m=1}^M \B_{\X_\nu^{(m)}}\left(u^{(m)},r^{(m)}\right).
\end{align*}
According to Theorem~\ref{thm:NewtonKantorovich}, we have to estimate 
\begin{align*}
\left\Vert D_i T^{(m)}(u)- D_i T^{(m)}(\bu)\right\Vert_{B(\X_\nu^{(i)},\X_\nu^{(m)})},
\end{align*}
for all $i,m=1,\ldots,M$ and $u\in \bbox(\bu,\rstar)$. We are thefore going to estimate, for an arbitrary  $h\in\X_\nu^{(i)}$, 
\begin{align*}
\left\Vert \left(D_i T^{(m)}(u)- D_i T^{(m)}(\bu)\right) h\right\Vert_{\X_\nu^{(m)}}.
\end{align*}

Using again the lower triangular structure of $T^\pprod$, we have $D_i T^{(m)}(u)= D_i T^{(m)}(\bu)=0$ for $i>m$. Otherwise, starting from the first derivative computations made in~\eqref{eq:DmTm}-\eqref{eq:DiTm}, denoting $v=u-\bu$, and proceeding as in~\eqref{eq:D2T}, we get
\begin{align*}
&\left( \left(D_m T^{(m)}(u)-D_m T^{(m)}(\bu)\right)h \right)(t)\\
& =  \int_{\tau_{m-1}}^{t} e^{ (\tau_m-s)\L^{(m)} } \sum_{j=0}^{2J-1}\D^j
	 \int_{0}^{1}  \left(g^{(j)}\right)'' \! \left( \bu^{(m)} \left(s\right) + \tilde{s} v^{(m)}\left(s\right)\right) \ast v^{(m)}(s) \ast h(s)  \ \mbox{d} \tilde{s} \ \mbox{d}s,
\end{align*}
and similarly for $i<m$,
\begin{align*}
&\left( \left(D_i T^{(m)}(u)-D_i T^{(m)}(\bu)\right)h \right)(t) \\
& = e^{ (t-\tau_{m-1})\L^{(m)}}\kL^{(m,i)} \! \int_{\tau_{i-1}}^{\tau_i} \! e^{ (\tau_i-s)\L^{(i)} } \sum_{j=0}^{2J-1}\D^j
	\! \int_{0}^{1} \! \left(g^{(j)}\right)'' \! \left( \bu^{(i)} \left(s\right) + \tilde{s}  v^{(i)}\left(s\right)\right) \ast v^{(i)}(s) \ast h(s)  \ \mbox{d} \tilde{s} \ \mbox{d}s.
\end{align*}
Taking the $C^0$ norm component-wise we get, for all $i\leq m$,
\begin{align*}
&\left\Vert \left(D_iT^{(m)}(u) - D_iT^{(m)}(\bu)\right)h \right\Vert_{C^0} \\
&\qquad \qquad \leq B^{(m,i)}\sum_{j=0}^{2J-1} \vert \D^j\vert  \left( \left \vert \left( g^{(j)} \right)'' \right \vert \left( \Vert \bu^{(i)} \Vert_{C^{0}} + \Vert v^{(i)} \Vert_{C^{0}} \right)  \ast \Vert v^{(i)} \Vert_{C^{0}} \ast \Vert h \Vert_{C^{0}}	\right).
\end{align*}
We then apply the $\ell^1_\nu$-norm, and estimate, using Lemma~\ref{lem:abs_pol},
\begin{align*}
&\left\Vert \left(D_iT^{(m)}(u) - D_iT^{(m)}(\bu)\right)h \right\Vert_{\X^{(m)}_\nu} \nonumber\\
& = \left\Vert\left\Vert \left(D_iT^{(m)}(u) - D_iT^{(m)}(\bu)\right)h \right\Vert_{C^0}\right\Vert_{\ell^1_\nu} \\
&\leq  \sum_{j=0}^{2J-1} \left\Vert B^{(m,i)}\vert \D^j\vert\right\Vert_{B(\ell^1_\nu,\ell^1_\nu)}   \left \vert \left( g^{(j)} \right)'' \right \vert \left( \Vert\Vert \bu^{(i)} \Vert_{C^{0}}\Vert_{\ell^1_\nu} +\Vert\Vert v^{(i)} \Vert_{C^{0}}\Vert_{\ell^1_\nu} \right)  \Vert\Vert v^{(i)} \Vert_{C^{0}}\Vert_{\ell^1_\nu} \left\Vert\left\Vert h \right\Vert_{C^{0}}\right\Vert_{\ell^1_\nu} \\
&\leq  \sum_{j=0}^{2J-1} \left\Vert B^{(m,i)}\vert \D^j\vert\right\Vert_{B(\ell^1_\nu,\ell^1_\nu)}   \left \vert \left( g^{(j)} \right)'' \right \vert \left( \Vert \bu^{(i)} \Vert_{\X^{(i)}_\nu} + \rstar^{(i)} \right) \Vert v^{(i)} \Vert_{\X^{(i)}_\nu} \Vert h \Vert_{\X^{(i)}_\nu}.
\end{align*}
Notice that $\left\Vert B^{(m,i)}\vert \D^j\vert\right\Vert_{\ell^1_\nu}$ can be computed explicitly, as 
\begin{align*}
\left\Vert B^{(m,i)}\vert \D^j\vert\right\Vert_{\ell^1_\nu} = \max\left(\left\Vert \Pi^{\leq \NL} B^{(m,i)}\vert \D^j\vert\Pi^{\leq \NL}\right\Vert_{B(\ell^1_\nu,\ell^1_\nu)} ,\, \chi^{(i,j)}_{N} \mu^{(m,i)}_{N}\right).
\end{align*}
Therefore
\begin{align*}
W^{(m)}_{i,k} \bydef 
\begin{cases}
\ds \sum_{j=0}^{2J-1} \left\Vert B^{(m,i)}\vert \D^j\vert\right\Vert_{B(\ell^1_\nu,\ell^1_\nu)}    \left \vert \left( g^{(j)} \right)'' \right \vert \left( \Vert \bu^{(i)} \Vert_{\X^{(i)}_\nu} + \rstar^{(i)} \right) \qquad \text{if }i\leq m \text{ and } k=i\\
0 \qquad \text{otherwise},
\end{cases}
\end{align*}
is computable and satisfies assumption~\eqref{e:def_W} of Theorem~\ref{thm:NewtonKantorovich}.

\subsection{Practical considerations and choices of parameters}
\label{sec:parameters}

The estimates $Y$, $Z$ and $W$ derived in the previous subsections depend on several parameters, and choosing these parameters carefully can be instrumental in getting sharp enough estimates so that~\eqref{e:inequalities1}-\eqref{e:inequalities2} holds. We discuss below the influence of the main parameters, some of which can be adapted automatically by the code. We emphasize that none of the steps and arguments described here have to be fully rigorous, as the goal here is only to guess appropriate parameters values, that will hopefully lead to a successful proof (and possibly to an error bound below some threshold), while minimizing the computation cost of the proof.

\paragraph{The choice of the subdivision.} It is well known for ODEs and some classes of PDEs that using adaptive step-sizes leads to much more efficient algorithms for numerically approximating solutions. In our context the same is true, but because our focus is on rigorously validating solutions, the adaptation of the time-step will be done differently. That is, the subdivision 
\begin{align*}
0 = \tau_0 < \tau_1 < \ldots < \tau_M = \tau,
\end{align*}
will not be chosen so as to minimize the error, but instead as a mean of making the $Z$ estimate smaller, which is usually the bottleneck of rigorous validation via Theorem~\ref{thm:NewtonKantorovich}.

Analyzing precisely how the subdivision influences all the $Z_i^{(m)}$ bounds is tricky, but what can easily be done is to study the diagonal terms $Z_m^{(m)}$. In order to explain the special relevance of these diagonal entries, let us look at condition~\eqref{e:inequalities2}. Taking $W=0$ for simplicity (in practice, the impact of $W$ is moderate as soon as the $r^j$ are small), there exists an $\eta \in \RR_{>0}^M$ such that~\eqref{e:inequalities2} holds if and only if the spectral radius of the matrix $Z$ is strictly less than one. Recalling that $Z$ is lower triangular for our problem, we see that making the diagonal entries of $Z$ smaller than one is necessary for satisfying~\eqref{e:inequalities2}. In practice, what we therefore do is to let the user pick a final time $\tau$ and a number $M$ of subdomains (depending on computational resources available and on how long one is willing to wait for the proof), and then the code automatically tries to select a subdivision for which  $Z_m^{(m)}$ is roughly uniform over $m$, so as to approximately minimize the spectral radius of~$Z$. Of course this minimization is not done rigorously, and we are free to use some heuristics to speed up the calculations of approximate $Z_m^{(m)}$ used for determining the subdomains. 

In some parts of the time domain the bounds may depend quite sensitively on the positions of the grid points, especially where the diagonalization of the linear operator is badly conditioned, see Remark~\ref{rem:diagonalization}. We improved the robusteness of the grid optimization code by not just selecting shorter subdomains in those parts of the time domain (which is a direct consequence of the uniformization of the $Z_m^{(m)}$ bound over all $m$), but also being pragmatic and aiming for smaller values of $Z_m^{(m)}$ there. We thus accept the resulting reduction in optimality of the grid. The goal is for the algorithm to produce a good grid for an efficient proof, not the optimal one.

An example of subdivision produced by this procedure, which was used to prove Theorem~\ref{th:OK}, is shown in Section~\ref{sec:OK}, Figure~\ref{fig:OKdomainsandK}.

\paragraph{Adapting the polynomial order in time to reduce the error bound.} As already mentioned, we do not select the subdivision so as to try to make the error bounds $r^j$ small. However, getting small errors and errors bounds is of course desirable, so the $Y$ bound is the next thing one can optimize for. Recall that the $Y$ bound controls how small the residual error $T(\bu)-\bu$ is. The subdivision is already fixed, but we can still influence the residual error by adapting the order, henceforth denoted~$K^{(m)}$, of the polynomial approximation in time used on each subdomain to compute the approximate solution $\bu^{(m)}$. Moreover, even if each $K^{(m)}$ has been chosen so as to obtain a very accurate approximate solution $\bu$, the $Y$ bound is only an upper-bound of the residual error, obtained using the procedure described in Section~\ref{sec:interp}. In order for this error bound to be relatively sharp we need accurate interpolation (estimates), hence a large enough $\tilde{K}^{(m)}$ should be used on each subdomain to bound $C^0$ norms as in~\eqref{eq:C0interp}. In practice, the user can provide a threshold for the $Y$ bound, and the code will then automatically try to select $K^{(m)}$ and $\tilde{K}^{(m)}$ that are as small as possible while having each $Y^{(m)}$ below that threshold.

\paragraph{Other important parameters.}
The last important parameters one may want to optimize for are truncation levels in Fourier space: both the number $N_u$ of Fourier modes used to construct the approximate solution $\bu$, and the number $\NL$ of modes kept in the finite part of $\L$. In principle, those could also depend on the subdomain $m$, but in practice we did not find a simple and easy-to-implement way of optimizing for those, so we decided to keep a single value of $N_u$ and $\NL$ over all subdomains, and to let the user select them. We emphasize that these two parameters have rather different impacts on the proof. The truncation level $N_u$ dictates how accurate our approximate solution is going to be, and therefore mainly impacts the $Y$ bound. In particular, note that in order to have a really small residual error, one must approximate the solution well enough in time as described just above, but one must also select $N_u$ large enough. On the other hand, $\NL$ controls how well each operator $\L^{(m)}$ approximates $DF(\bu^{(m)})$, and therefore impacts the $Z$ bound.

%!TEX root = main.tex

\section{Integrating all the way to infinity}
\label{sec:infinity}

In this section, we consider the specific problem of integrating an orbit that converges toward a stable steady state solution. 
More precisely, we focus on hyperbolic stable steady states, and show that our rigorous integrator is sufficiently robust to potentially integrate all the way to $t=+\infty$. Moreover, we show that, if this integration step to $t=+\infty$ is successful, then we automatically get in one go the existence of a stable steady state, together with an explicit neighborhood in which all solutions converge towards the steady state, and a spectral gap for the linearized operator.

In Section~\ref{sec:infinity_onedomain}, we show how the setup of Section~\ref{sec:onedomain} can be adjusted, with minor modifications, in order to deal with the unbounded time-interval $[\tau,\infty)$. In Section~\ref{sec:infinity_bounds}, we adjust accordingly the corresponding $Y$, $Z$ and $W$ bounds allowing to conduct a rigorous integration step on that unbounded time interval. In Section~\ref{sec:steadystate}, we then prove that this rigorous integration on $[\tau,\infty)$ automatically yields the existence of a steady state, its local asymptotic stability, and an explicit spectral gap. Finally, we briefly explain in Section~\ref{sec:infinity_DD} how these considerations on a single domain can be coupled with the time-subdivision introduced in Section~\ref{sec:domaindecomposition}.

\subsection{Setup for a single step}
\label{sec:infinity_onedomain}

We consider again the PDE~\eqref{eq:PDE}, but this time for $t\in[\tau,\infty)$, and with an initial condition at $t=\tau$: $u(\tau,x) = \uin(x)$. We are going to consider a time-independent approximate solution $\bu$ of the PDE on the interval $t\in[\tau,\infty)$, and try to prove that there exists an exact solution nearby. As we will see later on, this can only work if there is a stable steady state near $\bu$ (and if $\uin$ is close to $\bu$), but we emphasize that we do not need to know this a priori. Instead, it will be a consequence of our rigorous integration.

Our approach is the same as in Section~\ref{sec:onedomain}, with simplifications coming from that fact that the approximate solution $\bu$ is already time-independent. In particular, we define $\L$ exactly as in Section~\ref{sec:L_onedomain}, with $\bv^{(j)} = \Pi^{\leq \Nu} \left[\left(g^{(j)}\right)'(\bu)\right]$. In this section, we assume that all the eigenvalues of $\L$ have negative real part, which should be the case in practice if $\bu$ is indeed close to a hyperbolic stable steady state.

Since the time domain is now $[\tau,\infty)$, from now on
\begin{align*}
\X_\nu \bydef \left\{ u = \left(u_n\right)_{n\in\ZZ}\in \left(C([\tau,\infty),\CC)\right)^\ZZ,\ \left\Vert u\right\Vert_{\X_\nu} \bydef \sum_{n\in\ZZ} \left\Vert u_n\right\Vert_{C^0} \nu^{\vert n\vert} < \infty \right\},
\end{align*}
where the only difference with Definition~\ref{def:Xnu} is the time interval, and for any continuous function $v:[\tau,\infty) \to \CC$,
\begin{align*}
\left\Vert v\right\Vert_{C^0} \bydef \sup_{t\in [\tau,\infty)} \left\vert v(t)\right\vert.
\end{align*}
The fixed point operator $T$ is now given by
\begin{align}
\label{eq:T_L_infinity}
T(u)(t) = e^{(t-\tau)\L} \uin + \int_\tau^t e^{(t-s)\L} \gamma(u(s)) \d s,\qquad t\in[\tau,\infty),
\end{align}
which also only differs from~\eqref{eq:T_L} via the time domain. Note that the assumption that $\L$ only has eigenvalues with negative real part ensures that $T$ maps $\X_\nu$ into itself.

\subsection{Bounds}
\label{sec:infinity_bounds}

We now derive suitable bounds $Y$, $Z$ and $W$ for the operator $T$ defined in~\eqref{eq:T_L_infinity}. Compared to Section~\ref{sec:onedomain}, the only estimate which requires some small adaptations is the $Y$ bound. Indeed, in the $Y$ bound we are required to estimate $C^0$-norms in time (see Section~\ref{sec:Ybounds}), which, on a bounded time-interval, we accomplished using interpolation polynomials and interpolation error estimates as discussed in Section~\ref{sec:interp}. Now that the time-interval is infinite, we can no longer use such tools and instead have to bound these $C^0$-norms by hand. On the other hand, in the $Z$ and $W$ bounds all the involved quantities are already constant, and we only have to make a couple of trivial adaptions, by generalizing in a natural way quantities that were before only defined for bounded time-intervals.

Let us start with the $Y$ bound. Using that $\bu$ does not depend on time, we get, for all $t\in[\tau,\infty)$,
\begin{align}
\label{eq:Tu-u_cste}
\left(T(\bu)-\bu\right)(t) &= e^{(t-\tau)\L}\uin + \left(e^{(t-\tau)\L}-I\right)\L^{-1}\gamma(\bu) - \bu \nonumber\\
&= e^{(t-\tau)\L}\left(\uin+\L^{-1}\gamma(\bu)\right) - \left(\bu + \L^{-1}\gamma(\bu)\right).
\end{align}
Recall that $\L$ and $\gamma$ are defined such that the PDE~\eqref{eq:PDE} writes $\partial_t u = \L u + \gamma(u)$. Since we expect $\bu$ to be close to a steady state, we should have $\L\bu + \gamma(\bu) \approx 0$, and therefore $\bu + \L^{-1}\gamma(\bu)$ should be small (and so should $\uin + \L^{-1}\gamma(\bu)$, provided $\uin\approx \bu$). Using that $\L = Q\Lambda Q^{-1}$, with $\Lambda$ a diagonal matrix whose diagonal elements all have negative real part, we get the estimate
\begin{align}
\label{eq:Yinfty}
\left\Vert T(\bu)-\bu\right\Vert_{\X_\nu} \leq \left\Vert \vert Q\vert \left\vert Q^{-1}(\uin+\L^{-1}\gamma(\bu))\right\vert + \left\vert \bu + \L^{-1}\gamma(\bu)\right\vert \right\Vert_{\ell^1_\nu} \bydef Y,
\end{align}
which gives us a computable $Y$ bound satisfying~\eqref{e:def_Y_M1}.

In order to obtain a $Z$ bound, we can replicate the estimates used in Section~\ref{sec:Zbounds}, with minor adaptations due to the new time-domain. That is, compared to Definition~\ref{def:notation_Z}, we take
\begin{align*}
B\bydef \vert Q\vert  \left(-\Re(\Lambda)^{-1}\right) = \vert Q\vert  \left\vert\Re(\Lambda)^{-1}\right\vert,\qquad \chi^{(j)}_{\NL} \bydef \sup\limits_{\vert n\vert >\NL} \frac{\vert n\vert^j}{\vert \Re(\lambda_n)\vert} \quad\text{for }j=0,\ldots,2J-1,
\end{align*} 
and the $C^0$-norm in each $\beta^{(j)}$ becomes a simple absolute value since $\bu$ no longer depends on time. Then, Lemma~\ref{lem:fromZtoell1} and Proposition~\ref{prop:ell1forZ} apply, and we get
\begin{align}
\label{eq:Zinfty}
\left\Vert DT(\bu)\right\Vert_{B(\X_\nu,\X_\nu)} &\leq \left\Vert B \left\vert Q^{-1} D\gamma(\bu) \right\vert \right\Vert_{B(\ell^1_\nu,\ell^1_\nu)} \nonumber \\
&\leq \max\left( \left\Vert B\Gamma \, \Pi^{\leq \bN+\NL} \right\Vert_{B(\ell^1_\nu,\ell^1_\nu)}, \sum_{j=0}^{2J-1} \left\vert \beta^{(j)}\right\vert \chi^{(j)}_{\NL}\right) \bydef Z,
\end{align}
which gives us a computable $Z$ bound satisfying~\eqref{e:def_Z_M1}.

Finally, the $W$ bound remains exactly the same (with the new $B$ defined just above), and we take
\begin{align}
\label{eq:Winfty}
W \bydef \sum_{j=0}^{2J-1} \left\Vert B \, \vert\D^j\vert\right\Vert_{B(\ell^1_\nu,\ell^1_\nu)}   \left \vert \left( g^{(j)} \right)'' \right \vert \left( \left\Vert \bar u \right\Vert_{\ell^1_\nu} + \rstar \right),
\end{align}
which is computable and satisfies assumption~\eqref{e:def_W_M1}.

In summary, if there exists $r\in(0,\rstar)$ such that conditions~\eqref{e:inequalities1_M1} and~\eqref{e:inequalities2_M1} hold for the bounds $Y$, $Z$ and $W$ derived in this section, then the fixed point operator $T$ from~\eqref{eq:T_L_infinity} has a fixed point in $\B_{\X_\nu}(\bu,r)$, i.e., the solution $u^\star$ of the PDE~\eqref{eq:PDE} with initial condition $u^\star(\tau,\cdot) = \uin$ is defined for all $t\in[\tau,\infty)$ and remains bounded in $\ell^1_\nu$.

\subsection{Existence of a stable steady state}
\label{sec:steadystate}

We now show that, if we can perform a rigorous integration step on $[\tau,\infty)$, with the estimates derived in the previous subsection, then there exist a locally unique steady state nearby. Recalling that all the eigenvalues of $\L$ are assumed to have negative real part, $\L$ is invertible, which allows us to rewrite the steady state equation $\L u + \gamma(u) = 0$ as a fixed point problem $T^\stat(u) = u$, where $T^{\stat}:\ell^1_\nu\to\ell^1_\nu$ is defined as
\begin{align*}
T^\stat(u) = -\L^{-1}\gamma(u).
\end{align*}
In a loose sense, this operator $T^\stat$ is already contained in the operator $T$ from~\eqref{eq:T_L_infinity}, and this will allow us to deduce the local contractivity of $T^\stat$ from that of $T$.
\begin{theorem}
\label{thm:steadystate}
Let $\uin,\bu\in\ell^1_\nu$. Assume that all the eigenvalues of $\L$ have negative real part, and that, for the bounds $Y$, $Z$ and $W$ derived in Section~\ref{sec:infinity_bounds},  conditions~\eqref{e:inequalities1_M1} and~\eqref{e:inequalities2_M1} hold for some $r\in(0,\rstar)$. Then, there exists a unique steady state $u^\stat$ of the PDE~\eqref{eq:PDE} in $\B_{\ell^1_\nu}(\bu,r)$.
\end{theorem}
\begin{proof}
As we have seen in~\eqref{eq:Tu-u_cste}, if $u\in\X_\nu$ is time-independent (or equivalently, if $u\in\ell^1_\nu$), then for all $t\geq \tau$
\begin{align}
\label{eq:TtoTstat}
T(u)(t) &= e^{(t-\tau)\L}\left(\uin+\L^{-1}\gamma(u)\right) - \L^{-1}\gamma(u) \nonumber \\
&= e^{(t-\tau)\L}\left(\uin-T^\stat(u)\right) + T^\stat(u).
\end{align}
In particular, $T(u)(t) \underset{t\to\infty}{\longrightarrow} T^\stat(u)$. Moreover, since all the estimates involving the $\X_\nu$-norm are obtained by taking a supremum in time, we immediately get
\begin{align*}
\left\Vert T^\stat(\bu)-\bu\right\Vert_{\ell^1_\nu} \leq \left\Vert T(\bu)-\bu\right\Vert_{\X_\nu} \leq Y,
\end{align*} 
with $Y$ as in~\eqref{eq:Yinfty},
\begin{align*}
\left\Vert DT^\stat(\bu)\right\Vert_{B(\ell^1_\nu,\ell^1_\nu)} \leq \left\Vert DT(\bu)\right\Vert_{B(\X_\nu,\X_\nu)} \leq Z,
\end{align*} 
with $Z$ as in~\eqref{eq:Zinfty}, and similarly
\begin{align*}
\left\Vert DT^\stat(u)-DT^\stat(\bu)\right\Vert_{B(\ell^1_\nu,\ell^1_\nu)} \leq W \left\Vert u-\bu\right\Vert_{\ell^1_\nu},\qquad \text{for all } u\in\B_{\ell^1_\nu}(\bu,\rstar),
\end{align*} 
with $W$ as in~\eqref{eq:Winfty}. That is, the estimates $Y$, $Z$ and $W$ derived for the fixed point operator $T$ in~\eqref{eq:T_L_infinity}, corresponding to the evolution PDE~\eqref{eq:PDE} on the time interval $[\tau,\infty)$, also apply for the fixed point operator $T^\stat$ corresponding to steady states. We can therefore apply Theorem~\ref{thm:NewtonKantorovich} to $T^\stat$, which yields the existence of a unique steady state $u^\stat$ of the PDE such that $\left\Vert u^\stat - \bu \right\Vert_{\ell^1_\nu} \leq r$. 
\end{proof}

Unsurprisingly, we next show that the global solution $u^\star$ of the PDE~\eqref{eq:PDE} with initial condition $u^\star(\tau,\cdot) = \uin$ obtained as a fixed point of $T$ from~\eqref{eq:T_L} converges to the steady state $u^\stat$.
\begin{lemma}
\label{lem:convergencetosteadystate}
Under the assumptions of Theorem~\ref{thm:steadystate}, denoting $u^\star$ the fixed point of $T$ from~\eqref{eq:T_L_infinity}, we have that $u^\star(t)$ converges to $u^\stat$ in $\ell^1_\nu$ as $t\to\infty$.
\end{lemma}
\begin{proof}
Since $T(u^\star) = u^\star$, we have for all $t\geq \tau$
\begin{align*}
u^\star(t) = e^{(t-\tau)\L}\uin + \int_\tau^t e^{(t-s)\L} \gamma(u^\star(s))\d s.
\end{align*}
We recall that $u^\star(t)$ is bounded in $\ell^1_\nu$, as $\left\Vert u^\star(t) - \bu\right\Vert_{\ell^1_\nu} \leq \left\Vert u^\star - \bu\right\Vert_{\X_\nu} \leq r$, and that all the eigenvalues of $\L$ have negative real part, hence $u^\star(t)$ converges in $\ell^1_\nu$ as $t\to\infty$, to a limit $u^{\lim}$ in $\B_{\ell^1_\nu}(\bu,r)$. Since $u^\star$ solves the PDE~\eqref{eq:PDE} and converges to $u^{\lim}$, $u^{\lim}$ must be a steady state of the PDE, i.e., a fixed point of $T^\stat$. However, by Theorem~\ref{thm:steadystate} $u^\stat$ is the unique fixed point of $T^\stat$ in $\B_{\ell^1_\nu}(\bu,r)$, therefore $u^{\lim} = u^\stat$.
\end{proof}
\begin{remark}
For parabolic equations of the form~\eqref{eq:PDE}, even if the solution is defined for all positive times, it does not necessarily have to converge to an equilibrium. It could for instance converge to a periodic orbit, or even have way more complicated (e.g., chaotic) dynamics. However, our setup does not allow us to rigorously integrate all the way to infinity for arbitrary solutions, but only when all the eigenvalues of $\L$ have negative real part (this assumption is, for instance, crucial to obtain~\eqref{eq:Yinfty}), which can only happen if the solution converges to a hyperbolic stable steady state. 
\end{remark}
Next, we show that under very slightly strengthened assumptions (namely, that the inequality in~\eqref{e:inequalities1_M1} is strict), all initial conditions in an explicit neighborhood of $u^\stat$ lead to global solutions converging to $u^\stat$. 
\begin{theorem}
\label{thm:stability}
Let $\uin,\bu\in\ell^1_\nu$. Assume that all the eigenvalues of $\L$ have negative real part, and that, for the bounds $Y$, $Z$ and $W$ derived in Section~\ref{sec:infinity_bounds}, there exists $r_{\min}\in(0,\rstar)$ such that
\begin{align*}
Y + Z r_{\min} + \frac{1}{2}W r_{\min}^2 < r_{\min}, \\
Z + W r_{\min} < 1.
\end{align*}
Then, let $r_{\max}\in(r_{\min},\rstar]$ such that
\begin{align*}
Y + Z r_{\max} + \frac{1}{2}W r_{\max}^2 \leq r_{\max}, \\
Z + W r_{\max} < 1,
\end{align*}
and let $\varepsilon>0$ such that the two following conditions are satisfied
\begin{align}
\left(1-(Z+Wr_{\min})\right)^2 - 4 W\left\Vert \vert Q\vert\, \vert Q^{-1}\vert \right\Vert_{B(\ell^1_\nu,\ell^1_\nu)}\varepsilon >0,
\label{eq:epsilon1}\\
\frac{1-(Z+Wr_{\min}) - \sqrt{\left(1-(Z+Wr_{\min})\right)^2 - 4 W\left\Vert \vert Q\vert\, \vert Q^{-1}\vert \right\Vert_{B(\ell^1_\nu,\ell^1_\nu)}\varepsilon}}{2W} \leq r_{\max}-r_{\min}.
\label{eq:epsilon2}
\end{align} 
Then, for any initial condition $\vin$ in $\B_{\ell^1_\nu}(u^\stat,\varepsilon)$, the solution $v^\star$ of the evolution PDE~\eqref{eq:PDE} with initial condition $v^\star(\tau,\cdot) = \vin$ is defined for all $t\geq \tau$, and converges to $u^\stat$ as $t\to\infty$.
\end{theorem}
\begin{proof}
First note that the existence of a positive $\varepsilon$ satisfying~\eqref{eq:epsilon1}-\eqref{eq:epsilon2} is guaranteed by the fact that $Z + W r_{\min} < 1$. Another preliminary observation is that the assumptions of Theorem~\ref{thm:stability} ensure that Theorem~\ref{thm:steadystate} holds both for $r=r_{\min}$ and $r=r_{\max}$. That is, we have that the steady state $u^\stat$ belongs to the ball $\B_{\ell^1_\nu}(\bu,r_{\min})$, and that it is the unique steady state within the larger ball $\B_{\ell^1_\nu}(\bu,r_{\max})$.

In order to better track the dependencies with respect to the initial condition, we now denote $T_\uin$ the fixed point operator from~\eqref{eq:T_L_infinity} with initial condition $\uin$. We start by showing that, for all $\vin\in\B_{\ell^1_\nu}(u^\stat,\varepsilon)$, the operator $T_\vin$ has a fixed point near $u^\stat$. We are going to obtain this result by finding a $\delta>0$ (independent of $\varepsilon$) for which $T_\vin$ maps $\B_{\X_\nu}(u^\stat,\delta)$ into itself and is contracting on that ball.
According to~\eqref{eq:TtoTstat},
\begin{align*}
\left(T_\vin(u^\stat) - u^\stat\right)(t) = e^{(t-\tau)\L}\left(\vin - u^\stat\right) \qquad \text{for all }t\geq \tau,
\end{align*}
and using again that all the eigenvalues of $\L$ have negative real part, we get
\begin{align}
\label{eq:Yvin}
\left\Vert T_\vin(u^\stat) - u^\stat \right\Vert_{\X_\nu} \leq \left\Vert \vert Q\vert\, \vert Q^{-1}\vert \right\Vert_{B(\ell^1_\nu,\ell^1_\nu)} \varepsilon.
\end{align}
Then, noting that $DT_\vin(u)$ does not depend on $\vin$, we can control $ DT_\vin(u)$ using the estimates $Z$ and $W$ for $DT_\uin$ from~\eqref{eq:Zinfty}-\eqref{eq:Winfty}. More precisely, provided $u\in\B_{\X_\nu}(\bu,\rstar)$ we get
\begin{align*}
\left\Vert DT_\vin(u) \right\Vert_{B(\X_\nu,\X_\nu)} &= \left\Vert DT_\uin(u) \right\Vert_{B(\X_\nu,\X_\nu)} \\
&\leq \left\Vert DT_\uin(\bu) \right\Vert_{B(\X_\nu,\X_\nu)} + \left\Vert DT_\uin(u) - DT_\uin(\bu) \right\Vert_{B(\X_\nu,\X_\nu)} \\
&\leq Z + W \left\Vert u - \bu\right\Vert_{\X_\nu}.
\end{align*}
In particular, if $u\in \B_{\X_\nu}(u^\stat,\delta)$ with $\delta\leq \rstar -r_{\min}$, then $u\in\B_{\X_\nu}(\bu,\rstar)$ and
\begin{align}
\label{eq:Zvin}
\left\Vert DT_\vin(u) \right\Vert_{B(\X_\nu,\X_\nu)} 
&\leq Z + W \left(\left\Vert u^\stat - \bu\right\Vert_{\X_\nu} + \left\Vert u - u^\stat\right\Vert_{\X_\nu} \right) \nonumber\\
&\leq Z + W \left(r_{\min} + \delta \right).
\end{align}
Combining this estimate with~\eqref{eq:Yvin} and using the mean value inequality then shows that, for all $u\in \B_{\X_\nu}(u^\stat,\delta)$ with $\delta\leq \rstar-r_{\min}$,
\begin{align*}
\left\Vert T_\vin(u) - u^\stat \right\Vert_{\X_\nu} &\leq \left\Vert T_\vin(u^\stat) - u^\stat \right\Vert_{\X_\nu} + \left\Vert T_\vin(u) - T_\vin(u^\stat) \right\Vert_{\X_\nu} \\
&\leq \left\Vert \vert Q\vert\, \vert Q^{-1}\vert \right\Vert_{B(\ell^1_\nu,\ell^1_\nu)} \varepsilon + \left(Z + Wr_{\min} \right)\delta  + W \delta^2.
\end{align*}
Using again that $Z + Wr_{\min}< 1$, condition~\eqref{eq:epsilon1} ensures that there exists $\delta>0$ such that 
\begin{align}
\label{eq:delta}
\left\Vert \vert Q\vert\, \vert Q^{-1}\vert \right\Vert_{B(\ell^1_\nu,\ell^1_\nu)} \varepsilon + \left(Z + Wr_{\min} \right)\delta  + W \delta^2 \leq \delta.
\end{align}
A suitable $\delta$ satisfying~\eqref{eq:delta} is given by
\begin{align*}
\delta = \frac{1-(Z+Wr_{\min}) - \sqrt{\left(1-(Z+Wr_{\min})\right)^2 - 4 W\left\Vert \vert Q\vert\, \vert Q^{-1}\vert \right\Vert_{B(\ell^1_\nu,\ell^1_\nu)}\varepsilon}}{2W},
\end{align*}
for which we have $\delta\leq r_{\max}-r_{\min} \leq \rstar - r_{\min}$ thanks to~\eqref{eq:epsilon2}. In particular, $\B_{\X_\nu}(u^\stat,\delta)$ is included in $\B_{\X_\nu}(\bu,\rstar)$, and therefore~\eqref{eq:delta} shows that 
\begin{align*}
\left\Vert T_\vin(u) - u^\stat \right\Vert_{\X_\nu} \leq \delta\qquad \text{for all }u\in \B_{\X_\nu}(u^\stat,\delta),
\end{align*}
i.e., $T_\vin$ maps $\B_{\X_\nu}(u^\stat,\delta)$ into itself. Estimate~\eqref{eq:delta} also implies that
\begin{align*}
Z + Wr_{\min} + W\delta < 1,
\end{align*}
and then~\eqref{eq:Zvin} shows that, for all $u$ in $\B_{\X_\nu}(u^\stat,\delta)$,
\begin{align*}
\left\Vert DT_\vin(u) \right\Vert_{B(\X_\nu,\X_\nu)} \leq Z + W\left(r_{\min} + \delta\right) < 1.
\end{align*}
Therefore, $T_\vin$ is a contraction on $\B_{\X_\nu}(u^\stat,\delta)$, and has a unique fixed point $v^\star$ in $\B_{\X_\nu}(u^\stat,\delta)$. 

We now conclude with similar arguments as in Lemma~\ref{lem:convergencetosteadystate}. This $v^\star$ is the solution of the evolution PDE~\eqref{eq:PDE} with initial condition $v^\star(\tau,\cdot) = \vin$, which is therefore global in time, and converges as $t\to\infty$ to a steady state $v^{\lim}$ of the PDE~\eqref{eq:PDE}. Since $v^{\lim}$ belongs to $\B_{\ell^1_\nu}(u^\stat,\delta)$, with $\delta\leq r_{\max}-r_{\min}$, $v^{\lim}$ also belongs to $\B_{\ell^1_\nu}(\bu,r_{\max})$, and by uniqueness $v^{\lim}$ must be equal to $u^\stat$.
\end{proof}

\begin{remark}
Theorem~\ref{thm:stability} proves that $\B_{\ell^1_\nu}(u^\stat,\varepsilon)$ is included in the domain of attraction of~$u^\stat$, and does so by using only bounds we have already obtained during the rigorous integration. 
In particular, while the inequalities~\eqref{eq:epsilon1} and~\eqref{eq:epsilon2} allow for the selection of an explicit value for the radius $\varepsilon$, the center of the ball $\B_{\ell^1_\nu}(u^\stat,\varepsilon)$ is not readily accessible, and this ball does not even contain $\bu$ if $\varepsilon<r_{\min}$. 
However, if one is specifically interested in getting a verified domain of attraction of $u^\stat$, one can slightly adapt the estimates to obtain better information.
Before doing so, pinpointing the location of $u^\stat$ can be improved by setting $Y^\infty =   \left\Vert \bu + \L^{-1}\gamma(\bu)  \right\Vert_{\ell^1_\nu}$, and computing (for the same $Z$ and $W$ as before, since $\bu$ is considered fixed in this analysis)
\[
  r_{\min}^\stat = \frac{(1-Z)-\sqrt{(1-Z)^2-2Y^\infty W}}{W}.
\]
Assuming for simplicity that $\rstar \geq \frac{1-Z}{W} = r_{\max}$, the analysis 
in the proof of Theorem~\ref{thm:steadystate} then shows that $u^\stat \in 
\B_{\ell^1_\nu}(\bu,r_{\min}^\stat)$ is the unique stationary point inside 
$\B_{\ell^1_\nu}(\bu,r_{\max})$. This leads to somewhat improved estimates for $\varepsilon$ in Theorem~\ref{thm:stability} as well. 

Now consider $Y=Y_\uin$, defined in~\eqref{eq:Yinfty}, to be a bound that depends on the initial data $\uin$. Given a point $\uin$ for which we want to prove that it is in the domain of attraction of $u^\stat$, we then check that 
$Y_\uin < \frac{(1-Z)^2}{2W}$. If so, then we know from  Lemma~\ref{lem:convergencetosteadystate} that the global solution starting at~$\uin$ converges to $u^\stat$. Hence the set
\[
  \left\{ \uin \in \B_{\ell^1_\nu}(\bu,r_{\max}) \, : \, 
  Y_\uin < \frac{(1-Z)^2}{2W} \right\}
\]
is part of the domain of attraction of $u^\stat$. It is not difficult to check rigorously that some $\uin$ lies in this set by using the computable expressions for $Y=Y_\uin$, $Z$ and $W$. 
\end{remark}

Finally, let us prove that there is a spectral gap for the linearization of the PDE~\eqref{eq:PDE} at the equilibrium $u^\stat$.
\begin{theorem}
Repeat the assumptions of Theorem~\ref{thm:steadystate}, and let $F(u)\bydef \L u +\gamma(u)$ be the right-hand-side of the PDE~\eqref{eq:PDE}. Then, there exists $\alpha<0$ such that the spectrum of  $DF(u^\stat)$ is contained in $\{z\in\CC,\ \Re(z) < \alpha\}$.
\end{theorem}
\begin{proof}
First note that $DF(u^\stat)$ has compact resolvent, hence its spectrum is only composed of eigenvalues. Recalling that the eigenvalues of $\L$ are denoted $\left(\lambda_n\right)_{n\in\ZZ}$, we consider 
\begin{align*}
\ul\bydef \max_{n\in\ZZ} \Re(\lambda_n)<0.
\end{align*}
For any $\alpha\in(\ul,0)$, we know that the spectrum of $\L$ is to the left of $\alpha$, and we are going to prove, by a simple homotopy argument, that the same is still true for $DF(u^\stat)$ provided $\alpha<0$ is small enough. More precisely, we consider, for all $s\in[0,1]$,
\begin{align*}
\L_s \bydef \L + s \left(DF(u^\stat)) - \L \right),
\end{align*}
and show that, for some $\alpha\in(\ul,0)$ and for all $s\in[0,1]$, $\L_s$ has no eigenvalue is the vertical strip $\S_\alpha \bydef \{z\in\CC,\ \vert \Re(z)\vert \in[\alpha,0]\}$. Once this is proven, since all the eigenvalues of $\L_0 = \L$ are to the left of that strip, by upper semicontinuity of the spectrum~\cite[IV.3, Theorem 3.16]{Kat13} we will know that the same is true for $\L_1 = DF(u^\stat)$.

We are going to prove the existence of $\alpha<0$ such that, for all $a\in[\alpha,0]$, all $\mu\in\RR$ and all $s\in[0,1]$, $\L_s - (a+i\mu)I$ is invertible, thereby showing that $\L_s$ has no eigenvalue in $\S_\alpha$. We first rewrite
\begin{align*}
\L_s - (a+i\mu)I &= \L - (a+i\mu)I + s \left(DF(u^\stat)) - \L \right)\\
&= \left(\L - (a+i\mu)I\right) \left(I + s \left(\L - (a+i\mu)I\right)^{-1} \left(DF(u^\stat)) - \L \right)\right).
\end{align*}
Since $DF(u^\stat) - \L=D\gamma(u^\stat)$, it will be sufficient to prove that 
\begin{align*}
\left\Vert \left(\L - (a+i\mu)I\right)^{-1} D\gamma(u^\stat)\right\Vert_{B(\ell^1_\nu,\ell^1_\nu)}<1,
\end{align*}
in order to get that $\L_s - (a+i\mu)I$ is invertible for all $s\in[0,1]$. To that end, we introduce for all $a\in[\alpha,0]$,
\begin{align*}
B(a)\bydef \vert Q\vert  \left\vert\left(\Re(\Lambda)-a\right)^{-1}\right\vert, \qquad\text{and}\qquad \chi^{(j)}_{\NL}(a) \bydef \sup\limits_{\vert n\vert >\NL} \frac{\vert n\vert^j}{\vert \Re(\lambda_n) - a \vert} \quad\text{for }j=0,\ldots,2J-1,
\end{align*}
generalizing the quantities $B$ and $\chi^{(j)}_{\NL}$ used in Section~\ref{sec:infinity_bounds} for $a=0$. 
We then get 
\begin{align*}
\left\Vert \left(\L - (a+i\mu)I\right)^{-1} D\gamma(u^\stat) \right\Vert_{B(\ell^1_\nu,\ell^1_\nu)} &\leq \left\Vert \vert Q\vert \, \left\vert\left(\Lambda - (a+i\mu)\right)^{-1}\right\vert\, \left\vert Q^{-1} D\gamma(u^\stat) \right\vert \right\Vert_{B(\ell^1_\nu,\ell^1_\nu)} \\
&= \left\Vert B(a) \left\vert Q^{-1} D\gamma(u^\stat) \right\vert \right\Vert_{B(\ell^1_\nu,\ell^1_\nu)} \\
&\leq \left\Vert B(a) \left\vert Q^{-1} D\gamma(\bu) \right\vert \right\Vert_{B(\ell^1_\nu,\ell^1_\nu)} \\
&\qquad + \left\Vert B(a) \left\vert Q^{-1} \left(D\gamma(u^\stat) - D\gamma(\bu)\right) \right\vert \right\Vert_{B(\ell^1_\nu,\ell^1_\nu)},
\end{align*}
and proceeding as in Section~\ref{sec:infinity_bounds}, we end up with
\begin{align}
& \left\Vert \left(\L - (a+i\mu)I\right)^{-1} D\gamma(u^\stat) \right\Vert_{B(\ell^1_\nu,\ell^1_\nu)}  \nonumber \\
& \hspace*{3cm} \leq \max\left( \left\Vert B(a)\Gamma \, \Pi^{\leq \bN+\NL} \right\Vert_{B(\ell^1_\nu,\ell^1_\nu)}, \sum_{j=0}^{2J-1} \left\vert \beta^{(j)}\right\vert \chi^{(j)}_{\NL}(a)\right)  \nonumber \\
& \hspace*{4cm} + \sum_{j=0}^{2J-1} \left\Vert B(a) \, \vert\D^j\vert\right\Vert_{B(\ell^1_\nu,\ell^1_\nu)}   \left \vert \left( g^{(j)} \right)'' \right \vert \left( \left\Vert \bar u \right\Vert_{\ell^1_\nu} + \rstar \right) \left\Vert u^\stat - \bu \right\Vert_{\ell^1_\nu}.  \label{eq:condLs}
\end{align}
For $a=0$,
\begin{align*}
\left\Vert \left(\L - i\mu I\right)^{-1} D\gamma(u^\stat) \right\Vert_{B(\ell^1_\nu,\ell^1_\nu)} \leq Z + Wr < 1,
\end{align*}
and by continuity there exists $\alpha>0$ such that 
\begin{align*}
\left\Vert \left(\L - (a+i\mu)I\right)^{-1} D\gamma(u^\stat) \right\Vert_{B(\ell^1_\nu,\ell^1_\nu)} & < 1 \qquad \text{for all }a\in[-\alpha,0]. \qedhere
\end{align*}
\end{proof}
\begin{remark}
The proof of Theorem~\ref{thm:stability} provides us with a straightforward way of getting an explicit spectral gap. Indeed, the right-hand-side of~\eqref{eq:condLs} is computable, and decreasing with respect to $a$, therefore one could just find an explicit $a<0$ for which this expression is still less than $1$, and we then get that Theorem~\ref{thm:stability} holds with $\alpha=a$.
\end{remark}

\subsection{The case of multiple domains}
\label{sec:infinity_DD}

In Sections~\ref{sec:infinity_onedomain} to~\ref{sec:steadystate}, we only deal with a single integration step from $\tau$ to $\infty$, assuming the initial data was already close to a stable steady state. However, this final integration step can also easily be incorporated in the setup with multiple subdomains introduced in Section~\ref{sec:domaindecomposition}. That is, we can take $\tau_M=\infty$  in~\eqref{eq:subdivision}, and, provided $\bu^{(M)}$ is taken time-independent, use the same piece-wise construction of $\L$ as in Section~\ref{sec:setup_DD}, leading to the fixed point operator $T^{\pprod}$ defined in~\eqref{eq:defTprod}-\eqref{eq:deftheta}.

The corresponding $Y$, $Z$ and $W$ estimates are then obtained as in Section~\ref{sec:Ybounds_DD} to Section~\ref{sec:Wbounds_DD}, with the small modifications for the last domain introduced in Section~\ref{sec:infinity_bounds}. All the results of Section~\ref{sec:steadystate} are still valid in that case because, if Theorem~\ref{thm:NewtonKantorovich} holds for $M>1$, then in particular we have that
\begin{align*}
Y^M + Z_M^M r^M + \frac{1}{2}W_{M,M}^{M} \left(r^M\right)^2 \leq r^{M}, \\
Z_M^M \eta^M + W^M_{M,M} r^M \eta^M < \eta^M,
\end{align*}
i.e., the conditions~\eqref{e:inequalities1_M1}-\eqref{e:inequalities2_M1} hold for the last integration step by itself as assumed in Section~\ref{sec:steadystate}.

%!TEX root = main.tex
\section{Applications}
\label{sec:examples}

In this section, we present some applications of the validation procedure introduced in this paper, and discuss the obtained results. First, in Section~\ref{sec:SH}, we consider the Swift--Hohenberg equation, and revisit a specific orbit already validated in~\cite{BerBreShe24}. We show that the approach proposed in this paper significantly reduces the cost of the proof.
Next, we consider the Ohta--Kawasaki equation in Section~\ref{sec:OK}, for which we rigorously integrate a solution which passes through some kind of meta-stable state all the way to $t=+\infty$. Finally, we focus on the Kuramoto--Sivashinsky equation in Section~\ref{sec:KS}, for which we rigorously integrate near a periodic orbit. All the code used for the computer-assisted parts of the proofs is available at~\cite{integratorcode_new}.

\subsection{The Swift--Hohenberg equation}
\label{sec:SH}

We first consider the Swift--Hohenberg equation on a torus $\TT_L = \RR/L$ of size $L>0$:
	\begin{align}
		\label{eq:SH}
		\begin{cases}
			\dfrac{ \partial u }{ \partial t} = 
			- \left(\dfrac{ \partial^{2} }{ \partial x^{2}} +1\right)^2 u +\alpha u - u^3 ,
			& (t,x) \in (0,\tend] \times \TT_L, \\[2ex]
			u (0, x) = \uin ( x ), & x \in \TT_L,
	\end{cases}
	\end{align}
which is a well-known parabolic PDE leading to pattern formation. It was already used as an example in~\cite[Theorem 1.5]{BerBreShe24}, in which the following result was obtained.
\begin{theorem}
\label{th:SH}
Consider the Swift--Hohenberg equation~\eqref{eq:SH} with $\alpha = 5$, $L=6\pi$, $\tend=3/2$, and $\uin(x) = 0.4\cos\left(\frac{2\pi x}{L}\right)-0.3\cos\left(\frac{4 \pi x}{L}\right)$. Let $\bu=\bu(t,x)$ be the function represented in Figure~\ref{fig:SH}, and whose precise description in terms of Fourier-Chebyshev coefficients can be downloaded at~\cite{integratorcode}. Then, there exists a smooth solution $u$ of~\eqref{eq:SH} such that
\begin{align*}
\sup_{t\in [0,\tend]} \sup_{x\in \TT_L} \vert u(t,x)-\bu(t,x) \vert \leq 
4\times 10^{-8}.
\end{align*}
\end{theorem}
\begin{figure}
\centering
\includegraphics[width=0.49\linewidth]{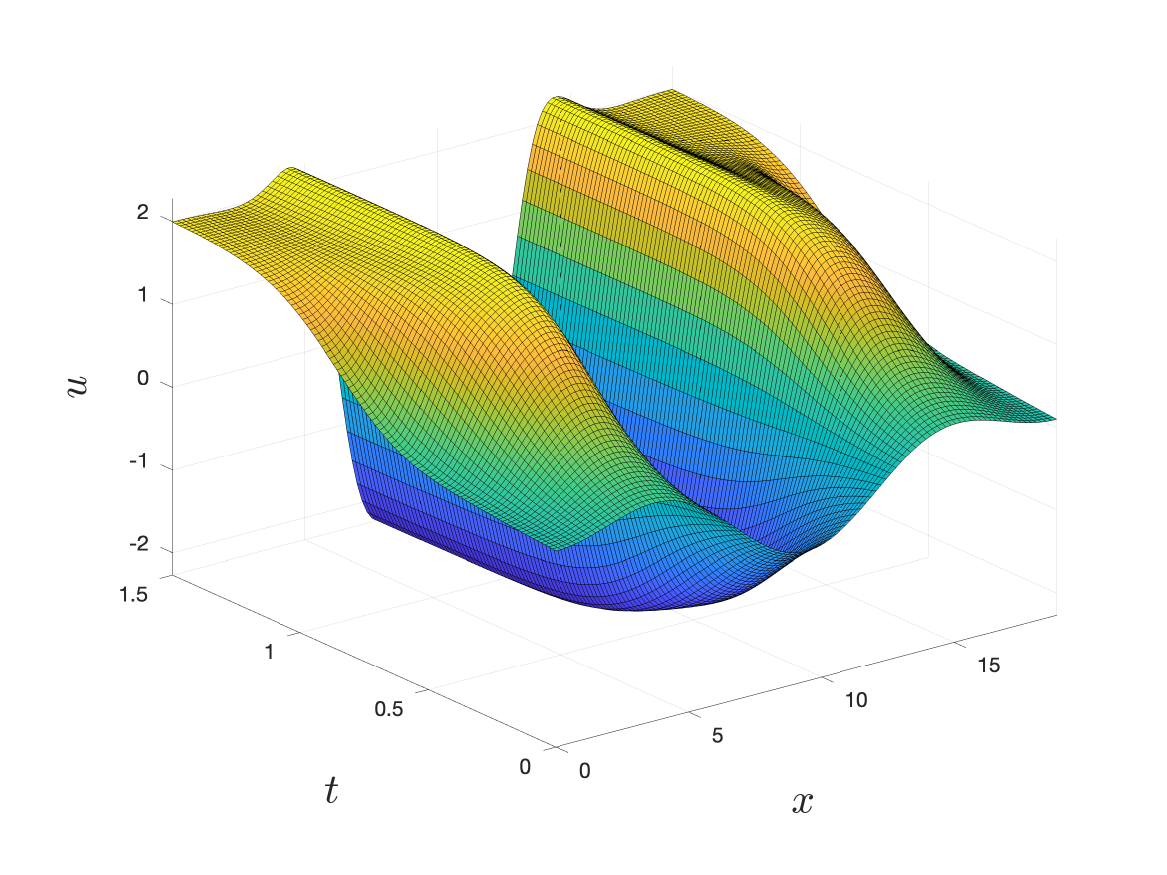}
\hfill
\includegraphics[width=0.49\linewidth]{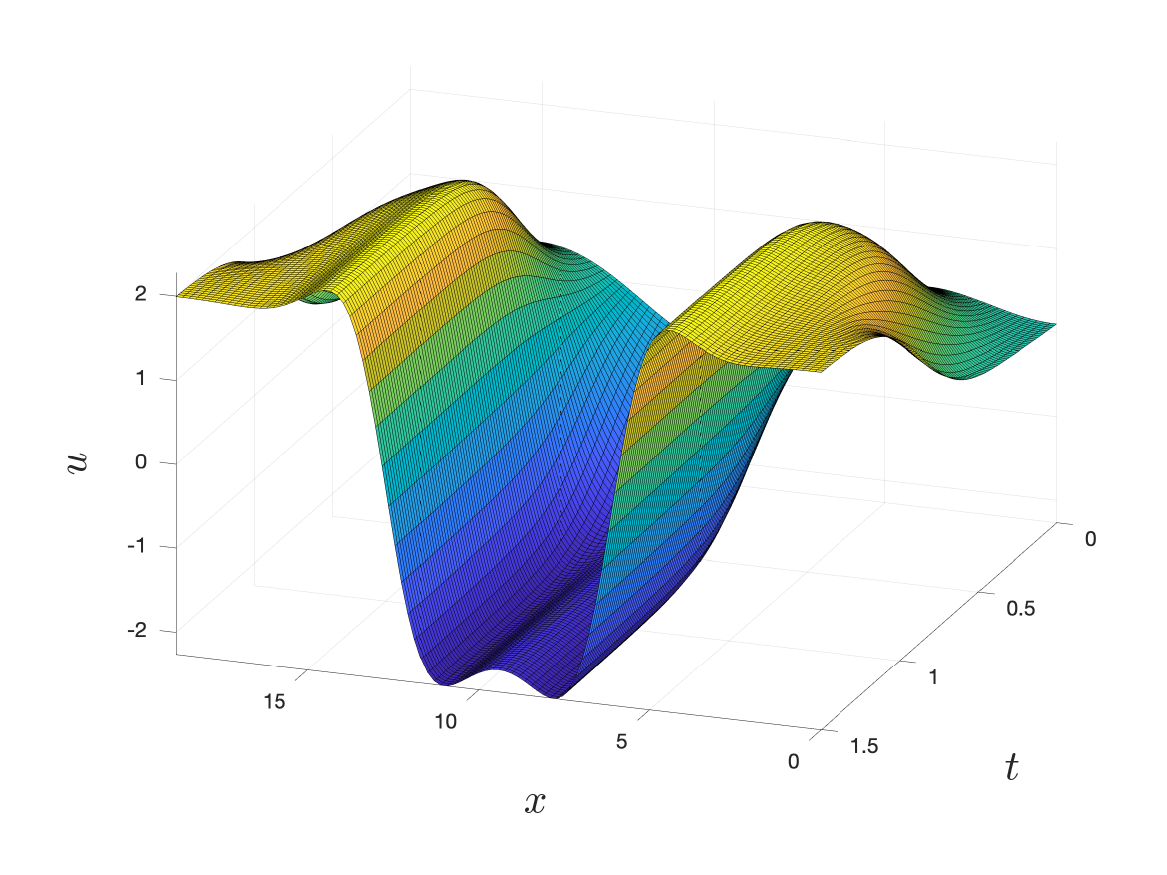}
\caption{The approximate solution $\bu$ of the Swift--Hohenberg equation~\eqref{eq:SH}, which has been validated in Theorem~\ref{th:SH}, depicted twice with different views.}
\label{fig:SH}
\end{figure}

We first show that, as claimed in the introduction, the approach introduced in the current paper is much more efficient than the one from~\cite{BerBreShe24}, as the removal of the approximate inverse allows for a much cheaper proof, while providing comparable error estimates. Indeed, keeping exactly the same approximate solution $\bu$ as in~\cite{BerBreShe24} (with $\Nu=30$ Fourier modes, $M=100$ subdomains, and polynomials of order $K=5$ in time on each subdomain) but using the estimates from Section~\ref{sec:domaindecomposition}, we obtain the following result.
\begin{theorem}
\label{th:SH_new}
Consider the Swift--Hohenberg equation~\eqref{eq:SH} with $\alpha = 5$, $L=6\pi$, $\tend=3/2$, and $\uin(x) = 0.4\cos\left(\frac{2\pi x}{L}\right)-0.3\cos\left(\frac{4 \pi x}{L}\right)$, as in Theorem~\ref{th:SH}, and let again $\bu=\bu(t,x)$ be the function represented in Figure~\ref{fig:SH}, and whose precise description in terms of Fourier-Chebyshev coefficients can be downloaded at~\cite{integratorcode}. Then, there exists a smooth solution $u$ of~\eqref{eq:SH} such that
\begin{align*}
\sup_{t\in [0,\tend]} \sup_{x\in \TT_L} \vert u(t,x)-\bu(t,x) \vert \leq 
2\times 10^{-9}.
\end{align*}
\end{theorem}
While the statements of Theorem~\ref{th:SH} and Theorem~\ref{th:SH_new} are identical except for a slightly different error bound, the key difference between the two is in their respective proofs (the proof of Theorem~\ref{th:SH_new} can be reproduced by running \runfile ~from~\cite{integratorcode_new}). Indeed the proof of Theorem~\ref{th:SH} is based on the setup from~\cite{BerBreShe24}, and required the computation of a matrix whose size is of the order or $\Nu M K\times\Nu M K$ for the approximate inverse. On the other hand, the larger objects that are required during the proof of Theorem~\ref{th:SH_new} are merely vectors whose size is of order $\Nu M K$ (for representing $\bu$, $\gamma(\bu)$, etc), and $2M$ matrices whose size is of the order of $\NL\times\NL$ (the matrices $Q^{(m)}$ and $\left(Q^{(m)}\right)^{-1}$ involved in the piece-wise definition of $\L$). Here we used $\NL = 30$, because this is what was done in~\cite{BerBreShe24} in which we always took $\Nu=\NL$ for simplicity, but $\NL$ could in fact be taken lower without much influence on the proof. To give an idea of the difference in efficiency, the proof using the setup from the current paper is roughly a factor 25 faster, using a factor 50 less memory and a factor 5 fewer cpus (but of course this depends on the hardware architecture).

It should also be noted that the removal of the approximate inverse not only leads to a much cheaper proof, but in fact also to smaller error bounds, as the setup introduced in this paper allows for sharper estimates compared to~\cite{BerBreShe24}. Moreover, we choose here to keep exactly the same subdomains and other computational parameters as in~\cite{BerBreShe24} in order to have a fair comparison, but using the strategies described in Section~\ref{sec:parameters} in order to adaptively select different subdomains and parameters one each subdomain would yield an even more efficient proof. We do take advantage of these strategies for the remaining examples, which were far out of reach of the methodology in~\cite{BerBreShe24}.

\subsection{The Ohta--Kawasaki equation}
\label{sec:OK}

We now come briefly back to the example of the Ohta-Kawasaki equation presented in Theorem~\ref{th:OK}, which has several notable features. In particular, as shown on Figure~\ref{fig:OK}, there is a kind of metastable regime during which the solution does not evolve much, before eventually converging to significantly different state. This metastable behaviour is not unexpected (see, e.g.~\cite{ChoMarWil11,CyrWan18,JohSanWan13,Wan16}), but it posed challenges from a numerical (and rigorous) integration perspective. In particular, if one had stopped the numerical simulation around $t=10$, one could have thought that the solution had converged towards a stable equilibrium, even though this is not at all the case. If Figure~\ref{fig:OK} was merely a numerical simulation, one could wonder whether a similar phenomenon might not occur again later on, i.e., whether the solution has in fact converged toward a stable equilibrium or not. However, Theorem~\ref{th:OK} provides us with a proof that the solution indeed converges to a steady state, which is very close to $u(\tau,\cdot)$, thanks to the analysis conducted in Section~\ref{sec:infinity}. 

We note that the PDE problem~\eqref{eq:OK} is posed with Neuman boundary conditions. Hence we double the domain size and work in the space of even periodic functions on this extended domain.
The solution in Figure~\ref{fig:OK} and Theorem~\ref{th:OK} is thus described by a cosine series with period $4\pi$ in space. As described in Remark~\ref{rem:symmetry}, the proof of integration to infinity in part relies on this even symmetry. This solution with domain restricted to $x\in[0,2\pi]$ then satisfies the Neuman boundary conditions. The stability statement of the limiting equilibrium within the class of even $4\pi$-periodic functions carries over to stability for the Neuman problem on $[0,2\pi]$. The additional symmetry $u(t,x)=u(t,2\pi-x)$, which is readily visible in Figure~\ref{fig:OK}, is not imposed or used in the proof. However, since the initial data satisfies this additional symmetry, it is conserved during the evolution.

\begin{figure}
\centering
\includegraphics[width=0.47\linewidth]{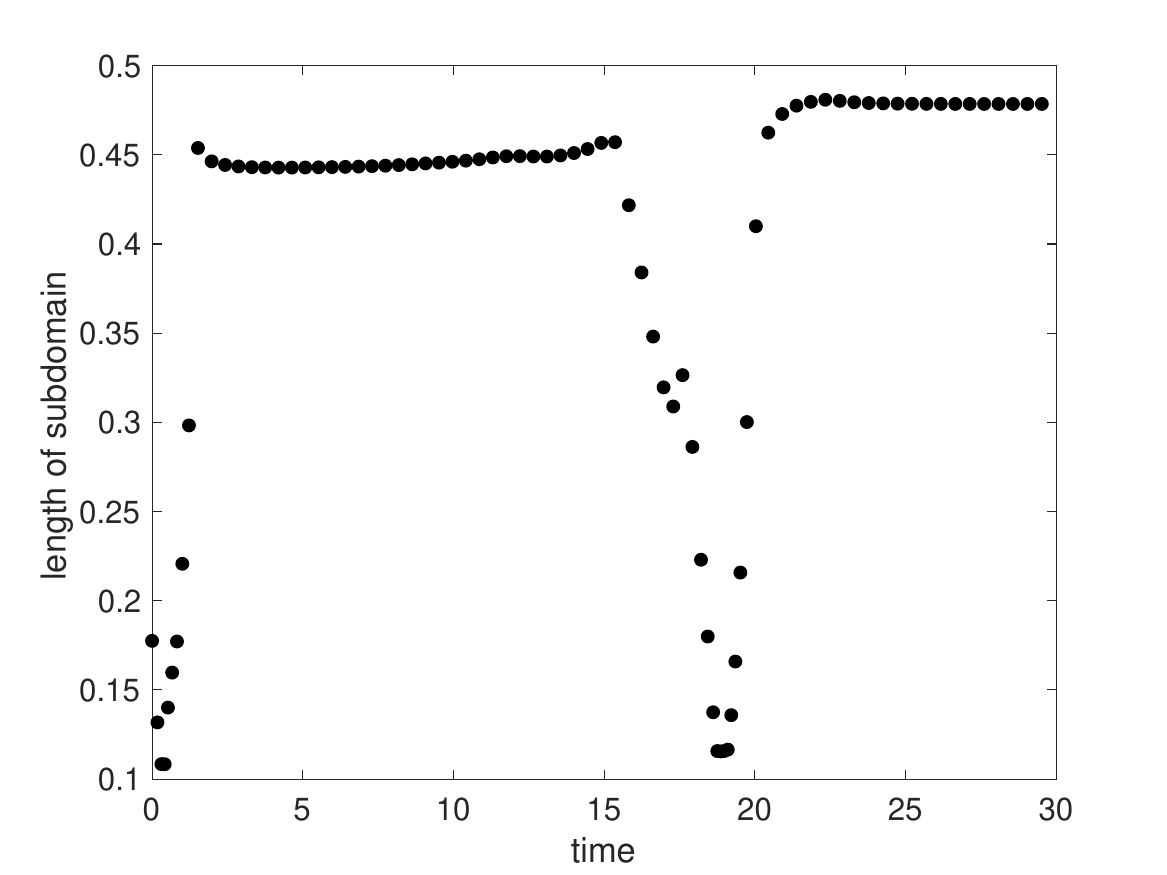}
\hfill
\includegraphics[width=0.47\linewidth]{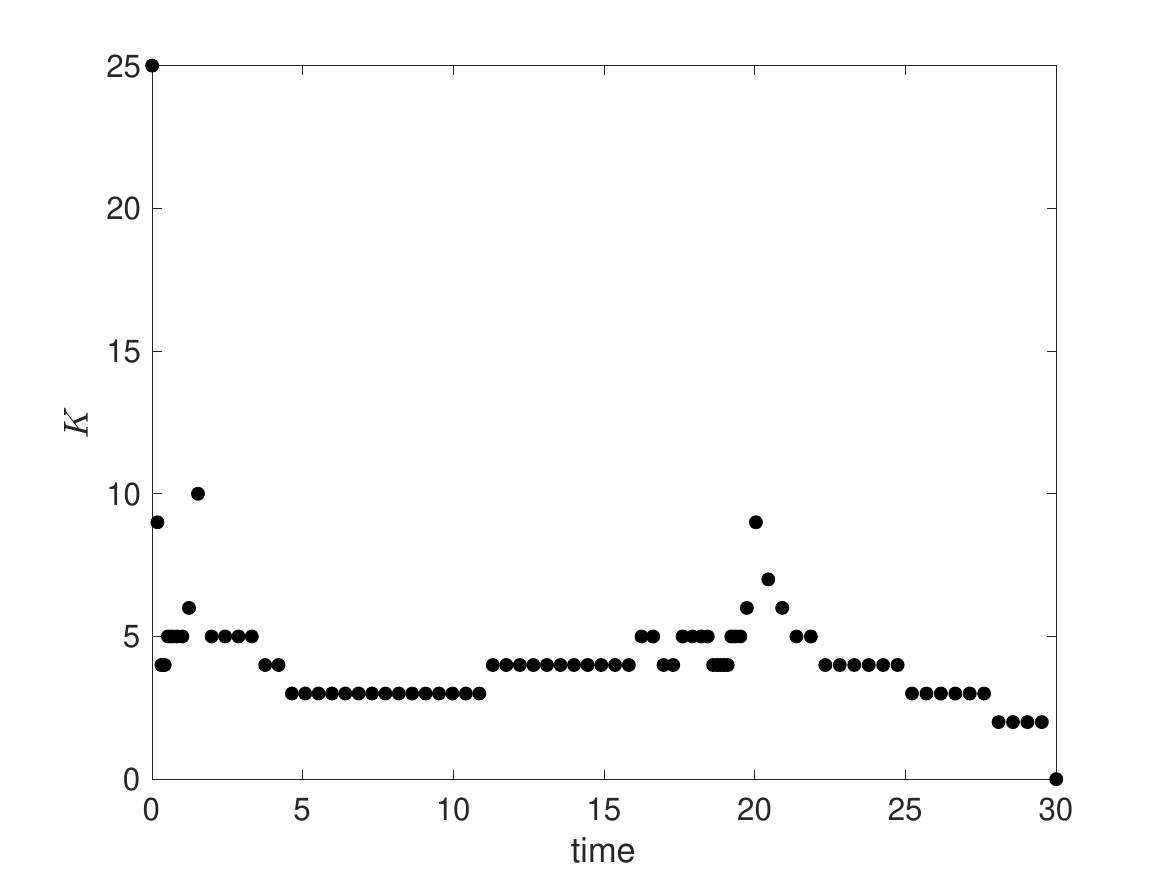}
\caption{Illustration of some of the important computational parameters used for the proof of Theorem~\ref{th:OK}. The length $\tau_{m+1}-\tau_{m}$ of each subdomain is shown on the left (except for the last one which has infinite length), and the Chebyshev order $K^{(m)}$ used for the approximate solution on each subdomain is shown on the right.}
\label{fig:OKdomainsandK}
\end{figure}

In order to prove Theorem~\ref{th:OK}, we used a non-uniform subdivision with $M=80$ subdomains represented on Figure~\ref{fig:OKdomainsandK}, obtained by following the procedure described in Section~\ref{sec:parameters}. The approximate solution $\bu$ has $\Nu=60$ Fourier modes, but the polynomial order in time $K^{(m)}$ depends on the subdomain, as also shown on Figure~\ref{fig:OKdomainsandK}. Figure~\ref{fig:OKdomainsandK} showcases that our adaptive subdomain algorithm behave as expected. In particular, when the solution does not evolve very rapidly, it automatically takes advantage of it by selecting longer time-steps. Once the subdomains are fixed, the values of $K^{(m)}$ selected and shown on Figure~\ref{fig:OKdomainsandK} are close to the minimal ones required to get each $Y^{(m)}$ below a given tolerance (here $10^{-8}$, when the $Y^{(m)}$ are computed non-rigorously, i.e., without interval arithmetic). For example, one may notice that the algorithm detects that we should select a relatively large value for $K^{(1)}$.

\subsection{The Kuramoto--Sivashinsky equation}
\label{sec:KS}

Finally, we consider the Kuramoto--Sivashinsky equation on the torus $\TT_L = \RR/L$ 
\begin{align}
		\label{eq:KS}
		\begin{cases}
			\dfrac{ \partial u }{ \partial t} = 
			-\dfrac{ \partial^{4} u}{ \partial x^{4}} -\dfrac{ \partial^{2} u}{ \partial x^{2}} -\displaystyle\frac{1}{2}\dfrac{ \partial }{ \partial x} u^2 ,
			& (t,x) \in (0,\tend] \times \TT_L, \\[2ex]
			u (0, x) = \uin ( x ), & x \in \TT_L,
	\end{cases}
\end{align}
where for our particular rigorous integration the initial data $\uin$ is  taken to be odd, a symmetry which is conserved by the PDE.
Compared to our previous examples, this equation contains nonlinear terms involving spatial derivatives, is not associated with a gradient flow, and is a well-known example of a parabolic PDE generating chaotic dynamics (as shown rigorously in~\cite{WilZgl20,WilZgl24}). 
The reason for the appearance of the number $\alpha$ in the next theorem will become apparent below. 

\begin{figure}
\centering
\includegraphics[scale=0.45]{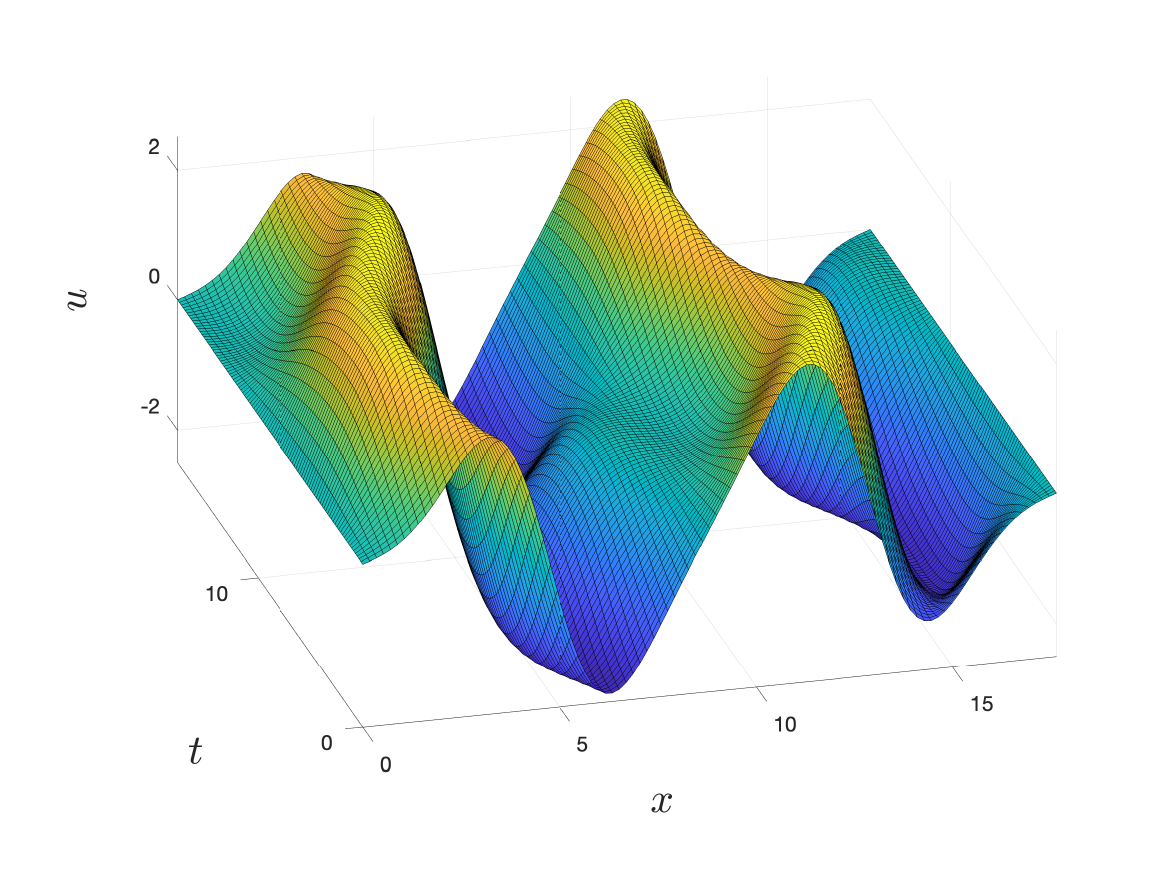}
\caption{The approximate solution $\bu$ of the Kuramoto--Sivashinsky equation~\eqref{eq:KS}, which has been validated in Theorem~\ref{th:KS}.}
\label{fig:KS}
\end{figure}

\begin{theorem}
\label{th:KS}
Let $\alpha=0.127$ and consider the Kuramoto–Sivashinsky equation~\eqref{eq:KS},
	with $L=2\pi\sqrt{\alpha}$, $\tend=2.245/\alpha$. Let $\bu=\bu(t,x)$ be the function represented in Figure~\ref{fig:KS}, and whose precise description in terms of Chebyshev$\times$Fourier coefficients can be downloaded at~\cite{integratorcode_new}, with corresponding initial data $\uin$ also available at~\cite{integratorcode_new}. Then, there exists a smooth solution $u$ of~\eqref{eq:KS} such that
\begin{align*}
\sup_{t\in [0,\tend]} \sup_{x\in \TT_L} \vert u(t,x)-\bu(t,x) \vert \leq 1\times 10^{-6}.
\end{align*}
\end{theorem}

The orbit validated in Theorem~\ref{th:KS} seems to be very near a stable time-periodic solution of the Kuramoto--Sivashinsky equation. We did not attempt to prove the existence of this nearby periodic orbit with our rigorous integrator (this would require a boundary value setup which is beyond the scope of this work, but close to what we plan to do for connecting orbits in ongoing work). However, we note that there exist several computer-assisted works allowing to directly study periodic orbits of the Kuramoto--Sivashinsky equation~\cite{Zgl10,GamLes17,FigLla17}. In particular, after the rescaling
\begin{align*}
u= -2\sqrt{\alpha}v,\quad t = \frac{1}{\alpha} s,\quad x = \frac{1}{\sqrt{\alpha}} y, 
\end{align*} 
the Kuramoto-Sivashinsky equation~\eqref{eq:KS} becomes
\begin{align*}
\dfrac{ \partial v }{ \partial s} = 
			-\alpha\dfrac{ \partial^{4} v}{ \partial y^{4}} -\dfrac{ \partial^{2} v}{ \partial y^{2}} +2v\dfrac{ \partial v}{ \partial y} , \quad (t,x) \in (0,\alpha\tend] \times \TT_{2\pi},
\end{align*}
and the orbit obtained in Theorem~\ref{th:KS} is close to one of the truly time-periodic orbits first obtained in~\cite{Zgl04}.

\section*{Acknowledgments}

MB is supported by the ANR project CAPPS: ANR-23-CE40-0004-01.

\bibliographystyle{abbrv}
\bibliography{bibfile} 

\end{document}